\documentclass[article]{amsart}
\usepackage[utf8]{inputenc}

\usepackage{enumerate}
\usepackage{fontenc}
\usepackage{color}
\usepackage{amsfonts,amsmath,amssymb,amsthm}
\usepackage{dsfont}
\usepackage{latexsym}
\usepackage{enumerate}
\usepackage[T1]{fontenc}
\usepackage{yfonts}
\usepackage{amsmath,amsfonts,amstext}
\usepackage{amssymb,amsmath,amscd,graphicx,epsfig,hhline,mathrsfs,yhmath,latexsym,url}
\usepackage[arrow, matrix, curve]{xy}
\usepackage[british]{babel}
\usepackage{paralist}               	
\usepackage[parfill]{parskip}			
\usepackage{tikz-cd}
\usepackage{tikz}

\newtheorem{theorem}{Theorem}[section]
\newtheorem{corollary}[theorem]{Corollary}

\newtheorem{definition}[theorem]{Definition}

\newtheorem{lemma}[theorem]{Lemma}
\newtheorem{proposition}[theorem]{Proposition}

\newtheorem{remarka}[theorem]{Remark}
\newenvironment{remark}{\begin{remarka}\rm}{\hfill\rule{2mm}{2mm}\end{remarka}}

\newtheorem*{T*}{Theorem}
\newtheorem*{A*}{Proposition}
\newtheorem*{Cor*}{Corollary}

\theoremstyle{definitionbreak}
\theoremstyle{definitionbreak}\newtheorem*{D*}{Definition}
\newtheorem{Ex}[theorem]{Example}

\newtheorem{prelem}{{\bf Theorem}}

\def\aff#1{{\rm Aff}(#1)}

\def\supp#1{{\rm sp}(#1)}

\title{A higher-order generalization of group theory}
\author{Bal\'azs Szegedy}

\begin{document}		
\maketitle

\begin{abstract} 
The goal of this paper is to show that fundamental concepts in higher-order Fourier analysis can be nauturally extended to the non-commutative setting. We generalize Gowers norms to arbitrary compact non-commutative groups. On the structural side, we show that nilspace theory (the algebraic part of higher-order Fourier analysis) can be naturally extended to include all non-commutative groups. To this end, we introduce generalized nilspaces called "groupspaces" and demonstrate that they possess properties very similar to nilspaces. We study $k$-th order generalizations of groups that are special groupspaces called {\it k-step} groupspaces. One step groupspaces are groups. We show that $k$-step groupspaces admit the structure of an iterated principal bundle with structure groups $G_1,G_2,\dots,G_k$. A similar, but somewhat more technical statement holds for general groupspaces, with possibly infinitely many structure groups. Structure groups of groupspaces are in some sense analogous to higher homotopy groups. In particular we use a version of the Eckmann-Hilton argument from homotopy theory to show that $G_i$ is abelian for $i\geq 2$. Groupspaces also show some similarities with $n$-groups from higher category theory (also used in physics) but the exact relationship between these concepts is a subject of future research. 
\end{abstract}

\section{introduction}

Fourier analysis is a powerful tool for studying functions on finite abelian groups and, more generally, on compact or locally compact abelian groups. A natural extension of Fourier analysis to non-commutative groups is representation theory, which is fundamental in various fields of mathematics and physics. In recent decades, driven by developments in additive combinatorics and ergodic theory, a deeper understanding of Fourier analysis has emerged, leading to a new theory known as Higher Order Fourier Analysis (HOF). Central to this theory are the Gowers norms $\|.\|_{U_k}$ \cite{Gowers_2001}, \cite{Gowers_2007},\cite{GTZ},\cite{Candela_2022}, which are used to detect "higher order resonances" in functions. 

A signature algebraic aspect of HOF is that it introduces a novel form of representation theory for abelian groups, where the target spaces of representations are derived from non-commutative algebraic structures such as nilmanifolds \cite{szeghof},\cite{CANDELA2019},\cite{CANDELA2023},\cite{Candela2023-II}. An intriguing interpretatioin of this phenomenon is that there is an emergent non-commutativity within high complexity functions on commutative groups. However, the non-commutativity observed in HOF is restricted to certain nilpotent structures that remain closely related to being abelian.

To obtain an axiomatic understanding of HOF, the algebraic theory of so-called nilspaces was developed \cite{Szegnil},\cite{Parallel},\cite{Candela-1},\cite{Candela-2},\cite{Gutman-1},\cite{Gutman-2},\cite{Gutman-3} . Nilspaces are common generalizations of abelian groups and nil-manifolds. In this framework, higher-order representations of abelian groups are defined as morphisms into compact finite dimensional nilspaces. It was subsequently shown that compact nilspaces can be constructed from double coset spaces of nilpotent groups \cite{double}. However this constiction is highly non-trivial and nilspace theory is still in development. 

The primary motivation for this paper comes from the observation that Gowers norms can be generalized for functions on arbitrary (not necessarily commutative) compact groups. We explain this construction in Chapter \ref{non-comm-gowers}. This observation hints at the existence of a higher order version of group theory that includes nilspaces and all non-commutative groups at the same time. Our paper puts down the foundations of such a theory. We introduce new structures called {\it groupspaces} through modifications of the nilspace axioms. We show that, although groupspaces are more general than nilspaces, many important properties of nilspaces generalize to them. 

Now we describe the main idea behind groupspaces at an informal level. In classical group theory, the structure of a group is defined by specific operations on its elements. Groupspaces, however, are characterized by particular arrangements of elements called {\it cubes}.

The simplest example for a groupspace arises from a group $G$. Here a $k$-dimensional cube is defined as a map $c:\{0,1\}^k\to G$ of the form $$c(v)=g_0g_1^{v_1}g_2^{v_2}\dots g_k^{v_k}$$ where $v_i$ is the $i$-th coordinate of $v$ and $g_0,g_1,\dots,g_k$ are elements in $G$. For instance, if $G=\mathbb{R}^2$ then $2$-dimensional cubes are parallelograms in $\mathbb{R}^2$. Generally $C^k(G)$ denotes the set of all $k$-dimensional cubes in $G$, which is a subset of $G^{\{0,1\}^k}$. The groupspace structure of $G$ is given by the infinite sequence $C^k(G)$, $k=0,1,2,\dots$. 

A general groupspace consists of a set $X$ along with a sequence of cubes $C^k(X)\subseteq X^{\{0,1\}^k}$ satisfying three combinatorial axioms. 

\begin{enumerate}

\item {\bf Presheaf axiom:} This axiom is a composition rule. It says that for certain "nice" maps (called cube morphisms) of the form $\phi:\{0,1\}^k\to\{0,1\}^n$ we have that if $c\in C^n(X)$ then $c\circ\phi\in C^k(X)$ holds. This is the axiom which separates groupspaces from nilspaces depending on the notion of nice maps. In case of groupspaces, fewer maps are considered to be cube morphisms and thus groupspaces are more general than nilspaces. For example the map $\phi:\{0,1\}^2\to\{0,1\}^2$ defined by $\phi((x,y))=(y,x)$ is a cube morphism in nilpsace theory but it is not a cube morphism in the definition of groupspaces. Roghly speaking, a map $\phi:\{0,1\}^k\to\{0,1\}^n$ is a cube morphism (in the groupspace setting) if each coordinate function of the form $v\mapsto \phi(v)_i$ depends on at most one coordinate of $v$ and furthermore this dependence must be monotonic in $i$. The monotonicity is omitted in the nilspace setting. 

\item {\bf Ergodicity axiom:} This axiom states that $C^1(X)=X^{\{0,1\}}$. In other words, every pair of elements in $X$ forms a $1$-dimensional cube. 

\item {\bf Completion axiom:} This axiom says that certain configurations (called corners) that are the unions of $k+1$ appropriately attached $k$-dimensional cubes, can be completed to a $k+1$-dimensional cube. For example any $3$ points in $\mathbb{R}^2$ can be completed to a parallelogram. If the completion is unique for some $k$ then we say that the groupspace is $k$-step. It can be shown that $k$-step groupspaces are completely determined by the set $C^{k+1}(X)$. It turns out that $1$-step groupsapces are groups with the above mentioned groupspace structure.  It is worth mentioning that the completion axiom is somewhat analogous to the Kan condition in simplicial homotopy theory.
\end{enumerate}

\medskip

\begin{center}
\begin{tikzpicture}

\coordinate (A) at (0,0,0);
\coordinate (B) at (2,0,0);
\coordinate (C) at (2,2,0);
\coordinate (D) at (0,2,0);
\coordinate (E) at (0,0,2);
\coordinate (F) at (2,0,2);
\coordinate (G) at (2,2,2);
\coordinate (H) at (0,2,2);

\draw[thick, dotted] (A) -- (B) -- (C) -- (D) -- cycle; 
\draw[thick] (E) -- (F) -- (G) -- (H) -- cycle; 
\draw[thick] (A) -- (E); 
\draw[thick] (B) -- (F);
\draw[thick] (A) -- (B);
\draw[thick] (A) -- (D);
\draw[thick, dotted] (C) -- (G);
\draw[thick] (D) -- (H);


\end{tikzpicture}

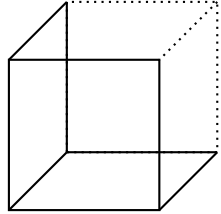
\text{\it Completion of a corner for $k=2$}

\end{center}

\medskip

Importantly, the precise definitions of these axioms are elementary and do not require prior knowledge of any deep theory. The paper is essentially self contained, so readers don't have to know higher-order Fourier analysis, nilspace theory, or category theory. Although it sometimes mentions basic category theory to enrich the context, one can ignore those parts without losing any understanding. 

The paper begins with an introductory section covering basic notations, definitions, and a chapter on affine homomorphisms between groups, which are slight generalizations of normal homomorphisms.
Chapter \ref{cubesin} introduces cubes in groups and describes their fundamental properties, including the equivalence between two different definitions of cubes (and their probability distributions), which proves to be a useful tool (see Corollary \ref{simpisgeneral}.) In Chapter \ref{non-comm-gowers}, Gowers norms on non-commutative groups are introduced. Although the proofs in this chapter are similar to the commutative case, Proposition \ref{samesure} from Chapter \ref{cubesin} is crucial for handling non-commutativity.

Chapters \ref{cubmor} and \ref{nil+group} define nilspaces and groupspaces, introducing morphisms between cubes (a combinatorial concept) and the three axioms. The concept of a $k$-step groupspace is also introduced. Chapter \ref{Class-1} characterizes all $1$-step groupspaces and shows that they are equivalent to groups. {\it In general we consider $k$-step groupspaces as $k$-th order generalizations of groups.} A remarkable aspect of this view point is the construction of higher degree structures on abelian groups (originally introduced in \cite{Szegnil}) and reviewed in Chapter \ref{highdegabgroups}. Using these structures one can for example consider classical polynomials in $\mathbb{R}[x]$ as morphisms between nilspaces (see Remark \ref{polyrem}).  Chapter \ref{exhigh} proves the existence of non-trivial higher-step groupspaces that are not nilspaces, showing that groupspace theory can produce new structures beyond groups and nilspaces. 

In later chapters we go deeper in the structures of groupsapces and generalize various useful facts from nilspace theory to them. We prove for every $k\in\mathbb{N}$ that every groupspace has a largest $k$-step factor (see Chapter \ref{charfact}). In some sense, these factors exploit the structure of every groupspace. Using these factors we show that every $k$-step groupspace $F$ admits a decomposition as an iterated principal group boundle with structure groups $G_i$ of the form

\[
\begin{tikzcd}[column sep=1.3cm]
    F_0 & \arrow[l, "G_1~\curvearrowright" above] F_1 & \arrow[l, "G_2~\curvearrowright" above] F_2 & \arrow[l, "G_3~\curvearrowright" above] \cdots & \arrow[l, "G_k~\curvearrowright" above] F_k=F
\end{tikzcd}
\]
where $F_0$ is a one point space and $G_k$ is abelian for $k\geq 2$. In this decomposition, $F_i$ is a unique maximal $i$-step factor of $F$. Note that the structure groups alone do not describe $F$ and for a complete description extra co-homological information on the extensions is needed. The fact that all but the first structure groups are abelian shows a surprising similarity with homotopy theory where higher homotopy groups are abelian. This similarity is further emphasized by the fact that the Eckmann-Hilton argument (see Theorem \ref{Eckmann-Hilton}) for the commutativity of higher homotopy groups is an important part of our proof.

While we do not give any specific application of groupspaces, we have good reasons to believe that they are natural and important structures. Nilspaces, that are special groupspaces, can be viewed as emergent structures in complicated functions on large abelian groups. This is also supported by the fact that nilspaces naturally appear as topological factors of ultraproducts of abelian groups. Additionally, Gowers norms can be generalized to compact nilspaces, and the densities of arithmetic configurations (like arithmetic progressions) can be calculated in subsets and functions on nilspaces.

Our results imply (however details for this are not given in this paper) that Gowers norms can be further extended to functions on compact groupspaces, making them one of the broadest familiy of structures where this is possible. This also implies that groupspaces provide new examples of higher-order Hilbert spaces, as introduced by Tao \cite{TaoBook}. Moreover, the densities of non-commutative configurations (represented by subsets of free groups) can be naturally calculated in functions on compact groupspaces.

Furthermore, we mention that the  type of definition used for groupsapces has proven useful in other contexts as well. For example, the algebraic structures emerging in the hypergraph regularity lemma (Delta complexes) share a similar nature. We believe that groupspaces fit into a broader category of higher-order generalizations of classical concepts, making them a valuable case study. Additionally, groupspaces are aesthetically pleasing structures, which in mathematics often correlates with usefulness.

Finally we mention another generalization of group theory which seems to be related to groupspaces, but the exact relationship is not clear. This is called an $n$-group or an $n$-dimensional higher group. The author of this paper was not able to find an elementary definition of this concept which does not use higher category theory or homotopy theory. The wikipedia says that {\it "The general definition of $n$-group is a matter of ongoing research. However, it is expected that every topological space will have a homotopy $n$-group at every point, which will encapsulate the Postnikov tower of the space up to the homotopy group $\pi_n$, or the entire Postnikov tower for $n=\infty$.} It seems however that $2$-groups are much better understood \cite{2groups}. Also the wikipedia says that: {\it 2-groups can be described using crossed modules and their classifying spaces.}

\medskip

\noindent{\bf Acknowledgement.} Research was supported by the NKFIH "\'Elvonal" KKP 133921 grant.

\section{Notation and basics}

We will use the short hand notation $[n]$ for the set $\{1,2,\dots,n\}$. If $X$ is a set then $X^n$ is the same as $x^{[n]}$ and for $x\in X^n$ we denote by $x_i\in X$ the $i$-th coordinate of $x$.

We will work with general abstract groups and also with compact topological groups. It will be important that if $G$ is a compact a group then there exits a unique left and right translation invariant Borel probability measure $\mu_G$ on $G$, called Haar measure. If $G$ is finite then $\mu_G$ is the uniform measure. For this reason we will sometimes refer to the Haar measure as the uniform measure on the compact group $G$. 

For statements where commutativity of a group is important we will use additiv notation $a+b$ for the group operation. In this context $na$ refers to the sum $a+a+\dots+a$ with $n$-terms. In general groups we use the conventional product notation $ab$. The notation $\prod_{i=1}^n$ in a group always denotes a left to right product:
$$\prod_{i=1}^n g_i=g_1g_2\dots g_n.$$

For conjugation in groups we follow the standard notation $g^h:=h^{-1}gh$. This is not to be confused with the power $g^n$ where $n\in\mathbb{N}$. We have that $(g_1g_2)^h=g_1^hg_2^h$, $(g^n)^b=(g^b)^n$ and that
$$g_1g_2\dots g_nh=hg_1^hg_2^h\dots g_n^h.$$ if $g_1,g_2,\dots,g_n,g,h\in G$ and $n\in\mathbb{Z}$.

We denote the free group with $k$ generators by $F_k=\langle g_1,g_2\dots,g_k\rangle$. We introduce the set $F_k^\square$ in the free group $F_k$ by $$F_k^\square:=\{\prod_{i=1}^k g_i^{v_i}~:~v\in\{0,1\}^k\}\subset F_k.$$ For every $k\in\mathbb{N}$ let $\tau_k:\{0,1\}^k\to F_k^\square$ denote the map defined by
\begin{equation}\label{tauk}
\tau_k(v)=\prod_{i=1}^kg_i^{v_i}.
\end{equation}
Observe that $\tau_k$ is a bijection and thus $\tau_k^{-1}$ exists. We say that $\tau_k$ is the {\it fundamental cube} in $F_k$.

\section{Affine homomorphisms between groups}

\begin{definition} Let $G_1$ and $G_2$ be groups. We say that $f:G_1\to G_2$ is an affine morphism if $f(x)=a_1 m(x) a_2$ holds for some $a_1,a_2\in G_2$ and group homomorphism $m:G_1\to G_2$. We denote the set of all affine morphisms from $G_1$ to $G_2$ by $\aff{G_1,G_2}$.
\end{definition}

Note that $a_1 m(x) a_2=a_1a_2m(x)^{a_2}=m(x)^{a_1^{-1}}a_1a_2$ holds. Since $x\mapsto m(x)^h$ is a group homomorphism this shows that affine morphisms can be equivalently defined by omitting one of $a_1$ and $a_2$. 

\begin{lemma}\label{charaff} Let $G_1,G_2$ be groups. A map $f:G_1\to G_2$ is an affine morphism if and only if for every quadruple  $a,b,c,d\in G_1$ with $ab^{-1}cd^{-1}=1$ we have that $f(a)f(b)^{-1}f(c)f(d)^{-1}=1$.
\end{lemma}

\begin{proof} If $f$ is an affine morphism then $f(x)=gm(x)$ holds for some group homomorphism $m:G_1\to G_2$ and $g\in G_2$. Thus $$f(a)f(b)^{-1}f(c)f(d)^{-1}=gm(a)m(b)^{-1}g^{-1}gm(c)m(d)^{-1}g^{-1}=$$
$$=(m(a)m(b)^{-1}m(c)m(d)^{-1})^{g^{-1}}=m(ab^{-1}cd^{-1})^{g^{-1}}=1.$$
To see the other direction let $f:G_1\to G_2$ be some function satisfying the condition in the statement. Our goal is to show that the map $g(x):=f(x)f(1)^{-1}$ is a homomorphism. Using the condition for  $1,a,ab,b$ we obtain that $1=f(1)f(a)^{-1}f(ab)f(b)^{-1}=g(a)^{-1}g(ab)g(b)^{-1}$ and so $g(a)g(b)=g(ab)$.
\end{proof}

Let $F_k=\langle g_1,g_2\dots,g_k\rangle$ denote the free group in $k$ generators and let $F_k^*$ denote the subset $\{1,g_1,g_2,\dots,g_k\}$ in $F_k$ . The next lemma shows that if $G$ is any group then the functions of the form $f:F_k^*\to G$ extend uniquely to affine morphisms and so they are in a natural one to one correspondence with the elements of $\aff{F_k,G}$.

\begin{lemma}\label{affchar} Let $G$ be an arbitrary group and $k\in\mathbb{N}$. Then 
\begin{enumerate}
\item every function $f:F_k^*\to G$ extends to a function in $\aff{F_k,G}$,
\item if two functions $f_1,f_2\in\aff{F_k,G}$ agree on $F_k^*$ then they are the same.
\end{enumerate}
\end{lemma}

\begin{proof} We start with the first claim. The universal property of the free group implies that there is a group homomorphism $h:F_k\to G$ such that hat $h(g_i)=f(1)^{-1}f(g_i)$ for every $1\leq i\leq k$. Then $f(1)h\in\aff{F_k,G}$ is an affine morphism extending $f$.

To prove the second claim we use that $f_1=a_1m_1$ and $f_2=a_1m_2$ holds for some $a_1,a_2\in G$ and group homomorphisms $m_1,m_2:F_k\to G$. Applying both equations to the identity element we obtain that $a_1=f_1(1)=f_2(1)=a_2$. Thus the group homomorphisms $m_1=a_1^{-1}f_1$ and $m_2=a_1^{-1}f_2$ agree on the generators $g_i$ for $1\leq i\leq k$. This implies that $m_1=m_2$ and so $f_1=f_2$.
\end{proof}

This lemma has the next corollary.

\begin{lemma}\label{lift} Let $G$ be a group, $N$ be a normal subgroup in $G$ and $k\in\mathbb{N}$. Then for every $f\in\aff{F_k,G/N}$ there is $f'\in\aff{F_k,G}$ such that $f=h\circ f'$ where $h:G\to G/N$ is the factor map. 
\end{lemma}

\begin{proof} Choose any $g:F_k^*\to G$ such that $h\circ g$ is equal to the restriction of $f$ to $F_k^*$. By lemma \ref{affchar} there is an extension $f':F_k\to G$ of $g$ such that $f'\in\aff{F_k,G}$. Now we have that $h\circ f'(x)$ agrees with $f(x)$ on $F_k^*$. Thus by lemma \ref{affchar} $h\circ f'=f$.
\end{proof}

\section{Cubes in non-commutative groups}\label{cubesin}

Cubes in Abelian groups are fundamental objects in Higher Order Fourier Analysis. If $A$ is an abelian group then a $k$-dimensional cube is a function of the form $c:\{0,1\}^k\to A$ defined by $$c(v):=a_0+\sum_{i=1}^k v_ia_i$$ where $a_0,a_1,\dots,a_k$ are elements in $A$. An alternative more symmetric definition is given by $$c(v):=\sum_{i=1}^k a_{i,v_i}$$ where $a_{i,j}$ for $i\in[k], j\in\{0,1\}$ are elements in $A$. It is easy to see that these two definitions lead to the same notion of a cube. The major difference is that while the first definition also determines a bijection between $k$-dimensional cubes and $A^{k+1}$, the second definition is redundant and different systems of elements can give the same cube. If $A$ is finite or more generally compact then both definitions can be used to define a "uniform" measure on cubes by choosing the elements $a_i$ or $a_{i,j}$ independently and uniformly (according to the Haar measure) at random. It is easy to see that these two probability measures are the same for every $k$. This leads to a robust notion of a random cube in $A$. The goal of this chapter is to show that similar statements hold in the non-commutative setting.


Let $G$ be an arbitrary group. We say that a {\it simple cube} in $G$ of dimension $k$ is a function $c:\{0,1\}^k\to G$ defined by
\begin{equation}\label{simplecubes}
c(v):=a_0\prod_{i=1}^ka_i^{v_i}
\end{equation}
where $a_0,a_1,\dots,a_k\in G$. If $G$ is compact with Haar measure $\mu_G$ and the elements $a_i$ are chosen independently with distribution $\mu_G$ then we obtain a random function $c$ whose distribution is denoted by $\nu_{G,k}$. The next lemma says that simple cubes are the images of the fundamental cube in the free group $F_k$ under affine homomorphisms in $\aff{F_k,G}$.

\begin{lemma}\label{simpaffdef} Let $k\in\mathbb{N}$, $G$ be a group and $c:\{0,1\}^k\to G$ be a function. Then $c$ is a simple cube if and only if there is an affine homomorphism $f:F_k\to G$ such that $c=f\circ\tau_k$ where $\tau_k$ is defined in (\ref{tauk}).
\end{lemma}

\begin{proof} Assume that $F_k=\langle h_1,h_2,\dots,h_k\rangle$. If $c$ is a simple cube defined as in (\ref{simplecubes}) then let $f:F_k\to G$ be the unique affine homomorphism  such that $f(1)=g_0$ and $f(h_i)=g_0g_i$ holds for $i\in [k]$. The existence of $f$ is guaranteed by lemma \ref{affchar}. Then it is easy to see that $f\circ\tau_k=c$. On the other hand if $f:F_k\to G$ is an affine homomorphism then let $g_0=f(1)$ and $g_i=f(1)^{-1}f(h_i)$ for $i\in [k]$. Let $c$ be as in (\ref{simplecubes}). Again we have that $f\circ\tau_k=c$ which completes the proof. 
\end{proof}

We say that a {\it general cube} of dimension $k$ is a function $c:\{0,1\}^k\to G$ defined by
$$c(v):=\prod_{i=1}^ka_{i,v_i}$$ where $a_{i,j}\in G$ for $i\in [k],j\in\{0,1\}$. If $G$ is compact and the elements $a_{i,j}$ are chosen independently with distribution $\mu_G$ then we obtain a random function $c$ whose distribution is denoted by $\nu'_{G,k}$. Both $\nu_{G,k}$ and $\nu'_{G,k}$ can be viewed as Borel probability measures on $G^{\{0,1\}^k}$. 

Now we show that simple cubes and general cubes are the same. It is clear that simple cubes are general cubes. Indeed: If a simple cube $c$ is given by the constants $a_0,a_1,\dots,a_k\in G$ then the general cube corresponding to the system $a_{1,0}=a_0,a_{1,1}=a_0a_1$ and $a_{i,0}=1,a_{i,1}=a_i$ for $2\leq i\leq k$ is equal to $c$. On the other hand, the next lemma implies that general cubes can be given as simple cubes.

\begin{lemma}\label{homeqlem} Let $G$ be a group, $n\in\mathbb{N}$ and $a_{i,j}, i\in [n], j\in\{0,1\}$ be a system of elements in $G$. Let $f:\{0,1\}^n\to G$ be the map defined by $f(v):=\prod_{i=1}^n a_{i,v_i}$. Then there are elements $c,h_1,h_2,\dots,h_n\in G$ such that $f(v)=c \prod_{i=1}^n h_i^{v_i}$. 
\end{lemma}

\begin{proof} We go by induction. If $n=1$ then $c=a_{1,0}, h_1=a_{1,0}^{-1}a_{1,1}$ is a good choice. Assume that the statement is true for $n-1\geq 1$. It follows that there are elements $c',h_1',h_2',\dots,h_{n-1}'\in G$  such that $\prod_{i=1}^{n-1} a_{i,v_i}=c'\prod_{i=1}^{n-1}h_i'^{v_i}$ holds for every $v\in\{0,1\}^{n-1}$. Now we have that
$$f(v)=c'\prod_{i=1}^{n-1}h_i'^{v_i}a_{n,v_n}=c'a_{n,0}\left(\prod_{i=1}^{n-1}h_i'^{v_i}\right)^{a_{n,0}}(a_{n,0}^{-1}a_{n,1})^{v_n}.$$ Consequently the choice $c:=c'a_{n,0},h_i:=h_i'^{a_{n,0}}$ for $i\in [n-1]$ and $h_n:=a_{n,0}^{-1}a_{n,1}$ satisfies the required equation.
\end{proof}

\begin{corollary}\label{simpisgeneral} Simple cubes and general cubes are the same objects. We will call them cubes.
\end{corollary}

In the rest of this chapter we show that the two types of probability measures $\nu_{G,k}$ and $\nu'_{G,k}$ on cubes are the same for every $k\in\mathbb{N}$. We will need the following elementary fact about the uniform (Haar) measure on compact groups. 

\begin{lemma}\label{conjindep} Let $G$ be a compact group. Assume that $a_1,a_2,\dots,a_n,b$ are independent and uniform elements in $G$. Then $a_1^b,a_2^b,\dots,a_n^b,b$ are also independent and uniform in $G$.
\end{lemma}

\begin{proof} From the left and right invariance of the Haar measure we have that the map $g\to g^b$ preserves the Haar measure for every fix $b$. It follows that for every fixed $b$ the random variables $a_1^b,a_2^b,\dots,a_k^b$ are independent and uniformly distributed in $G$. Since $b$ is chosen uniformly the proof is complete.
\end{proof} 

\begin{proposition}\label{samesure} We have $\nu_{G,k}=\nu'_{G,k}$ for every compact group $G$ and $k\in\mathbb{N}$.
\end{proposition}

\begin{proof} We go by induction on $k$. If $k=1$ then the statement says that if $a,b$ are two independent, uniform random element in $g$ then $a,ab$ are also independent and uniform. This follows from the property of the Haar measure that for every fixed $a$ the distribution of $ab$ is uniform. Assume that the statement is proved for $k-1\geq 1$. Then by induction we have that a $\nu'_{G,k}$ random cube $c$ is given by 
$$c(v)=a_0\prod_{i=1}^{k-1}a_i^{v_i}ba_k^{v_k}=ba_0^b\prod_{i=1}^{k-1}(a_i^b)^{v_i}a_k^{v_k}$$ where $a_0,a_1,\dots,a_k,b$ are chosen uniformly at random. Note here we use that $b,ba_k$ are two independent random elements in $G$ which is the case $k=1$. Now lemma \ref{conjindep} implies that conjugation with $b$ can be omitted from the above formula without changing the probability distribution. This completes the proof. 
\end{proof} 

\section{Uniformity norms on non-commutative groups}\label{non-comm-gowers}

The goal of this chapter is to introduce Gowers norms for functions on arbitrary finite groups or more generally compact topological groups. The definition is very similar to the Gowers norms: We evaluate the function on the vertices of a randomly chosen "cube" in the group $G$, we take the product of the function values with appropriate conjugations and then we take approprite power of the expectation of this product. 

\begin{definition}\label{non-comm-gow}  Let $G$ compact topological group with Haar measure $\mu_G$. Let $f:G\to\mathbb{C}$ be a bounded Borel measurable function on $G$ and let $n\in\mathbb{N}$. Then 
$$\|f\|_{U_n}:=\Bigl(\int_{a_0,a_1,\dots,a_n\in G} \prod_{v\in\{0,1\}^n}J^{h(v)}f(a_0a_1^{v_1}a_2^{v_2}\dots a_n^{v_n})~d\mu_G^{n+1}\Bigr)^{1/2^n}$$ where $J$ is the complex conjugation and $h(v):=\sum_{i\in [n]} v_i$ is the hight function on $\{0,1\}^n$. 
\end{definition}

To prove that $\|.\|_{U_n}$ is a norm on $L^\infty(G)$ for $n\geq 2$ we need the following non-commutative generalization of the Gowers inner product.

\begin{definition}\label{non-comm-prod}  Let $G$ be a compact topological group with Haar measure $\mu_G$. Let $n\in\mathbb{N}$ and let $\mathcal{F}=\{f_v:G\to\mathbb{C}\}_{v\in\{0,1\}^n}$ be a system of bounded Borel measurable function on $G$. Then their cubic product is defined as
$$(\mathcal{F})_{U_n}:=\int_{a_0,a_1,\dots,a_n\in G} \prod_{v\in\{0,1\}^n}J^{h(v)}f_v(a_0a_1^{v_1}a_2^{v_2}\dots a_n^{v_n})~d\mu_G^{n+1}$$ where $J$ and $h$ are as in definition \ref{non-comm-gow}.
\end{definition}

Observe that
\begin{equation}\label{alternativ-prod}
(\mathcal{F})_{U_n}=\int_{c\in G^{\{0,1\}^n}}\prod_{v\in\{0,1\}^n}J^{h(v)}f_v(c_v)~d\nu_{G,n}
\end{equation}
where $\nu_{G,n}$ is the uniform measure on $n$-dimensional cubes in $G$.

For a bounded Borel function $f:G\to\mathbb{C}$ let $[f]_n$ denote the function system $\{f_v=f\}_{v\in\{0,1\}^n}$ where each function is equal to $f$. With this notation we have that
$$\|f\|_{U_n}=([f]_n)_{U_n}^{1/2^n}.$$ For $i\in[n]$ and $j\in\{0,1\}$ let $q_{i,j}:\{0,1\}^n\to\{0,1\}^n$ denote the "face projection" which sets the $i$-th coordinate to $j$ and leaves the other coordinates unchanged. More precisely this means that $q_{i,j}(v)_i=j$ and $q_{i,j}(v)_s=v_s$ for $s\neq i$. Related to these operators we introduce the operators $Q_{i,j}$ with $i\in[n],j\in\{0,1\}$ acting on function systems of the form $\mathcal{F}=\{f_v\}_{v\in\{0,1\}^n}$ by 
$$Q_{i,j}(\mathcal{F})=\{g_v:=f_{q_{i,j}(v)}\}_{v\in\{0,1\}^n}.$$ 

Let us also introduce one co-dimensional faces of $\{0,1\}^n$ by $L_{i,j}=\{v:v\in\{0,1\}^n,v_i=j\}$ where $i\in[n]$ and $j\in\{0,1\}$. 

\begin{theorem}\label{innerprop} Let $\mathcal{F}=\{f_v:G\to\mathbb{C}\}_{v\in\{0,1\}^n}$ be a system of bounded Borel measurable functions on $G$ and let $n\geq 2$ be a natural number and $d\in[n],r\in\{0,1\}$. Then 
\begin{enumerate}
\item $(Q_{d,r}(\mathcal{F}))_{U_n}\geq 0$.
\item $|(\mathcal{F})_{U_n}|\leq (Q_{d,0}(\mathcal{F}))_{U_n}^{1/2}(Q_{d,1}(\mathcal{F}))_{U_n}^{1/2}$.
\item $|(\mathcal{F})_{U_n}|\leq\prod_{v\in\{0,1\}^n}\|f_v\|_{U_n}$.
\end{enumerate}
\end{theorem}

\begin{proof} By lemma \ref{samesure} and (\ref{alternativ-prod}) we have that $$(\mathcal{F})_{U_n}=\int_{G^{[n]\times\{0,1\}}} \prod_{v\in\{0,1\}^n}J^{h(v)}f_v(\prod_{i=1}^na_{i,v_i})~d\mu_G^{2n}.$$
\begin{equation}\label{gowinform2}
=\int_{G^{([n]\setminus\{d\})\times\{0,1\}}}\Bigl(\prod_{s\in\{0,1\}}\int_{a_{d,s}}\prod_{L_{d,s}}J^{h(v)}f_v(\prod_{i=1}^na_{i,v_i})~d\mu_G\Bigr)~d\mu_G^{2n-2}.
\end{equation}
This implies that 
\begin{equation}\label{gowinform3}
(Q_{d,r}(\mathcal{F}))_{U_n}=\int_{G^{([n]\setminus\{d\})\times\{0,1\}}}\Bigl|\int_{a_{d,r}}\prod_{L_{d,r}}J^{h(v)}f_v(\prod_{i=1}^na_{i,v_i})~d\mu_G\Bigr|^2~d\mu_G^{2n-2}\geq 0
\end{equation}
and thus the proof of the first claim is complete. Notice that equations (\ref{gowinform2}) and (\ref{gowinform3}) together with the Cauchy-Schwartz inequality imply the second claim. The third claim follows by iterating the second one for every $d\in[n]$. 
\end{proof}

\begin{lemma}[Seminorm property] For every $n\geq 2$ natural number and compact group $G$ and bounded Borel functions $f,g:G\to\mathbb{C}$ we have that
\begin{enumerate}
\item $\|f\|_{U_n}\geq 0$ 
\item $\|f+g\|_{U_n}\leq\|f\|_{U_n}+\|g\|_{U_n}$
\item $\|cf\|_{U_n}=\|f\|_{U_n}|c|$ for every $c\in\mathbb{C}$.
\end{enumerate}
\end{lemma}

\begin{proof} The first statement follows from the fact that $[f]_n=Q_{1,0}([f]_n)$ and theorem \ref{innerprop} by
$$\|f\|_{U_n}=([f]_n)_{U_n}^{1/2^n}=(Q_{1,0}([f]_n))_{U_n}^{1/2^n}\geq 0.$$
The second and the third statement follows from the multi-linear property of the form $(.)_{U_k}$. The second statement follows from $([cf]_n)_{U_n}=c^{2^n}([f]_n)_{U_n}$. The third statement follows from the decomposition of $([f+g]_n)_{U_n}$ into $2^{2^n}$ terms using multi-linearity. Each term is of the form $(\{h_v\}_{v\in\{0,1\}^n})_{U_n}$ with $h_v\in\{f,g\}$ for every $v$ and each possible configuration arises exactly once. By applying the third statement of theorem \ref{innerprop} we obtain that
$$\|f+g\|_{U_n}^{2^n}=([f+g]_n)_{U_n}\leq\sum_{i=0}^{2^n}{{2^n}\choose{i}}\|f\|_{U_n}^i\|g\|_{U_n}^{2^n-i}=(\|f\|_{U_n}+\|g\|_{U_n})^{2^n}.$$
\end{proof}

\begin{lemma}[Monotonicity] For $2<n$ natural numbers and bounded Borel function $f:G\to\mathbb{C}$ we have that $$\|f\|_{U_{n-1}}\leq\|f\|_{U_n}.$$
\end{lemma}

\begin{proof} Equation (\ref{gowinform3}) with $d=n,r=0$ and the Cauchy-Schwartz inequality implies that
$$(Q_{n,0}(\mathcal{F}))_{U_n}^{1/2}\geq\int_{G^{[n-1]\times\{0,1\}}}\int_{a_{n,0}}\prod_{L_{n,0}}J^{h(v)}f_v(\prod_{i=1}^na_{i,v_i})~d\mu_G~d\mu_G^{2n-2}=$$
$$\int_{a_{n,0}}\int_{G^{[n-1]\times\{0,1\}}}\prod_{L_{n,0}}J^{h(v)}f_v(\prod_{i=1}^na_{i,v_i})~d\mu_G^{2n-2}~d\mu_G.$$ 
Notice that if $\mathcal{F}=[f]_n$ then the inner integral is equal to $\|f\|_{U_{n-1}}^{2^{n-1}}$ for every fixed $a_{n,0}\in G$. We obtain in this case that $$\|f\|_{U_n}^{2^{n-1}}=(Q_{n,0}([f]_n))_{U_n}^{1/2}\geq\|f\|_{U_{n-1}}^{2^{n-1}}.$$ This completes the proof.
\end{proof}

\begin{definition} Let $f,g:G\to\mathbb{C}$ be a bounded Borel functions. Then we denote by $f\star g$ the function on $G$ whose value at $x$ is equal to $\int_{y\in G}f(xy)\overline{g}(y)~d\mu_G$.
\end{definition}

\begin{lemma} Let $f:G\to\mathbb{C}$ be a bounded Borel Function. Then $$\|f\|_{U_2}=\|f\star f\|_2^{1/2}.$$
\end{lemma}

\begin{proof} If we apply formula (\ref{gowinform3}) for $\mathcal{F}=[f]_2$ and $d=2,r=0$ we obtain that
$$\|f\|_{U_2}^4=\int_{a_{1,0},a_{1,1}\in G}\Bigl|\int_{a_{2,0}\in G} f(a_{1,0}a_{2,0})\overline{f}(a_{1,1}a_{2,0})~d\mu_G\Bigr|^2~d\mu_G^2=$$
$$=\int_{a_{1,0},a_{1,1}\in G}\Bigl|(f*f)(a_{1,0}a_{1,1}^{-1})|^2~d\mu_G^2=\|f\star f\|_2^2.$$
\end{proof}

\begin{corollary} Let $f:G\to\mathbb{C}$ be a bounded Borel Function. The $\|f\|_{U_2}=0$ if and only if $f$ is almost everywhere $0$.
\end{corollary}

\begin{corollary} We have for every natural number $n\geq 2$ that $(L^\infty(G),\|.\|_{U_n})$ is a normed space.
\end{corollary}

\section{Morphisms between discrete cubes}\label{cubmor}

We will use the short hand notation $\{0,1\}^n$ for $\{0,1\}^{[n]}$. We define $\{0,1\}^0$ as a one point set. For a cube $\{0,1\}^n$ let $p^n_i:\{0,1\}^n\to\{0,1\}$ denote the projection to the $i$-th coordinate. 

The goal of this chapter is to introduce two types of morphisms ($\mathcal{N}$ and $\mathcal{G}$) between cubes of the form $\{0,1\}^n$. Informally speaking, a map $\phi:\{0,1\}^n\to\{0,1\}^m$ is a morphism in $\mathcal{N}$ if for every $i\in [m]$ there is some $j\in [n]$ such that the $i$-th coordinate of $\phi(v)$ depends only on the $j$-th coordinate of $v$. Since this dependence is given by a function from $\{0,1\}$ to $\{0,1\}$ we have that there are $4$ possible dependences: The constant $0$ function, the constant $1$ function, the identity function and the $x\mapsto 1-x$ function. Note that if $v\mapsto\phi(v)_i$ is a constant function then $j$ is not uniquely determined by $i$ (it can be chosen to be any element in $[n]$). Morphisms in $\mathcal{G}$ are special morphisms in $\mathcal{N}$ in which $j$ is a monotonic function of $i$ if it is uniquely determined. The formal definition is the following.

\begin{definition}\label{cubmorph} A map $\phi:\{0,1\}^n\rightarrow\{0,1\}^m$ is a morphism in the category $\mathcal{N}$ if for every $i\in [m]$ we have that $p^m_i\circ\phi$ is either a constant function or there is $j\in[n]$ such that either $p^m_i\circ\phi=p^n_j$ or $p^m_i\circ\phi=1-p^n_j$. For such a morphism let us introduce the function $\gamma_\phi:[m]\to\mathbb{N}$ with $\gamma_\phi(i):=0$ if $p^m_i\circ\phi$ is a constant function and $\gamma_\phi(i):=j$ if $p^m_i\circ\phi=p^n_j$ or $p^m_i\circ\phi=1-p^n_j$. A map $\phi:\{0,1\}^n\rightarrow\{0,1\}^m$ is a morphism in the category $\mathcal{G}$ if $\phi$ is a morphism in $\mathcal{N}$ and $\gamma_\phi$ is monotonic on its support i.e. for every pair $1\leq i<j\leq m$ we have that either $\gamma_\phi(j)=0$ or $\gamma_\phi(i)\leq\gamma_\phi(j)$.
\end{definition}

As an example note that the map given by $\phi(x_1,x_2):=(0,x_1,x_2,1,x_1)$ is a morphism in $\mathcal{N}$ but it is not a morphism in $\mathcal{G}$. On the other hand the map $\phi'(x_1,x_2):=(0,x_1,x_1,1,x_2)$ is a morphism in both $\mathcal{N}$ and $\mathcal{G}$. For $n,m\in\mathbb{N}$ we denote by $\hom_{\mathcal{N}}(\{0,1\}^n,\{0,1\}^m)$ (resp. $\hom_{\mathcal{G}}(\{0,1\}^n,\{0,1\}^m)$) the set of morphisms in $\mathcal{N}$ (reps. $\mathcal{G}$). To check that both $\mathcal{N}$ and $\mathcal{G}$ are categories we need the following lemma. The proof is straightforward from the definitions. 

\begin{lemma} Let $\mathcal{C}$ be either $\mathcal{N}$ or $\mathcal{G}$. Then identity map $I_n:\{0,1\}^n\to\{0,1\}^n$ is in $\mathcal{C}$ for every $n\in\mathbb{N}$. If $\phi_1\in\hom_{\mathcal{C}}(\{0,1\}^n,\{0,1\}^m)$ and $\phi_2\in\hom_{\mathcal{C}}(\{0,1\}^m,\{0,1\}^k)$ then $\phi_2\circ\phi_1\in \hom_{\mathcal{C}}(\{0,1\}^n,\{0,1\}^k)$.
\end{lemma}

We continue with a list of a few important morphisms between cubes.

\noindent{\bf Face embeddings:}~~For natural numbers $n\geq 1$ and $i\in [n],j\in\{0,1\}$ let $e^n_{i,j}:\{0,1\}^{n-1}\to\{0,1\}^n$ denote the embedding defined by $$e^n_{i,j}(x_1,x_2,\dots,x_{n-1}):=(x_1,x_2,\dots,x_{i-1},j,x_i,x_{i+1},\dots,x_{n-1}).$$
We have that $e^n_{i,j}\in\hom_{\mathcal{G}}(\{0,1\}^{n-1},\{0,1\}^n)$. 

\medskip

\noindent{\bf Simplicial embeddings:}~~For a vector $v\in\{0,1\}^n$ let $\supp{v}$ denote the support of $v$ i.e. the set of coordinates $i$ for which $v_i=1$. For $v\in\{0,1\}^n$ let $s_v$ denote the map from $\{0,1\}^{h(v)}$ to $\{0,1\}^n$ defined by the property that $p^n_i\circ s_v=0$ if $i\notin \supp{v}$ and $p^n_i\circ e^n_v:=p^{h(v)}_j$ if $i\in \supp{v}$ and $i$ is the $j$-th smallest element in $\supp{v}$. This definition guarantees that $s_v$ is a morphism in $\mathcal{G}$. 

\medskip

\noindent{\bf Projections:}~~We introduced projections to individual coordiantes at the beginning of this chapter. A generalization of this concept works for an arbitrary subset of coordinates. Let $T\subseteq [n]$ with $T=\{a_1,a_2,\dots,a_m\}$ where $a_1<a_2<\dots<a_m$. Let $p_T:\{0,1\}^n\to\{0,1\}^m$ be the map defined by $p_T(v)_i=v_{a_i}$. It is clear that $p_{[n]\setminus\{i\}}\circ e^n_{i,j}$ is the identity map on $\{0,1\}^{n-1}$. We have that $p_T\in\hom_{\mathcal{G}}(\{0,1\}^{n},\{0,1\}^m)$.

\medskip

\noindent{\bf Reflections:}~~Let $r_i:\{0,1\}^n\to\{0,1\}^n$ be the cube morphism which flips the $i$-th coordinate to the opposite value. We have that 
$$r_i(x_1,x_2,\dots,x_n):=(x_1,x_2,\dots,x_{i-1},1-x_i,x_{i+1},\dots,x_n).$$ It is clear that $r_i\in\hom_{\mathcal{G}}(\{0,1\}^n,\{0,1\}^n)$. The next lemma about refelctions shows that $\{0,1\}^n$ is transitive even in the category $\mathcal{G}$.

\begin{lemma}[Symmetry of cubes]\label{transim} Let $n\in\mathbb{N}$ be fixed. Then the group $R_n$ generated by the reflections $\{r_i\}_{i=1}^n$ is isomorphic to $(\mathbb{Z}/2\mathbb{Z})^n$ and acts transitively on $\{0,1\}^n$. Every element of $R_n$ is in $\hom_\mathcal{G}(\{0,1\}^n,\{0,1\}^n)$.
\end{lemma}

\begin{proof} Let $f_0,f_1:\{0,1\}\to\{0,1\}$ be the functions $f_0(x)=x$ and $f_1(x)=1-x$ for $x=0,1$. The elements of $R_n$ are described by $$\phi_v(x_1,x_2,\dots,x_n)=(f_{v_1}(x_1),f_{v_2}(x_2),\dots,f_{v_n}(x_n))$$ where $v$ runs through $\{0,1\}^n$. Clearly every such map is in $\hom_\mathcal{G}(\{0,1\}^n,\{0,1\}^n)$ and $\phi_v\circ\phi_w=\phi_m$ where $m\equiv v+w\mod 2$. Transitivity follows from the fact that $f_v(0^n)=v$ holds for every $v\in\{0,1\}^n$.
\end{proof}

\medskip

\noindent{\bf Transpositions:}~~Let $t_{i,j}:\{0,1\}^n\to\{0,1\}^n$ be the map switching the $i$-th and $j$-th coordinates. This map is not in the category $\mathcal{G}$ but it is in $\mathcal{N}$. Their importance is that transpositions together with maps in $\mathcal{G}$ generate all maps in $\mathcal{N}$.

\medskip

Now wedescribe group theoretic interpretations of our categories $\mathcal{G}$ and $\mathcal{N}$.

\begin{proposition}\label{abhomeq} A map $\phi:\{0,1\}^n\rightarrow\{0,1\}^m$ is a morphism in $\mathcal{N}$ if and only if it extends to an affine morphism from $\mathbb{Z}^n$ to $\mathbb{Z}^m$.
\end{proposition}

\begin{proof} Assume that $\phi$ is a morphism in $\mathcal{N}$. Then for every $i\in [m]$ we have that $p^m_i\circ \phi$ is either a constant function or is equal to either $p^n_j(v)$ or $1-p^n_j(v)$. In each case $p^m_i\circ\phi$ extends to an affine morphism $f_i$ from $\mathbb{Z}^n$ to $\mathbb{Z}$. Then the map $f(x):=(f_1(x),f_2(x),\dots,f_m(x))$ from $\mathbb{Z}^n$ to $\mathbb{Z}^m$ is an extension of $\phi$ to an affine morphism. 

To see the other direction assume that $\phi:\{0,1\}^n\to\{0,1\}^m$ is the restriction of some affine morphism. This means that there is an integer matrix $M\in\mathbb{Z}^{n\times m}$ and an integer vector $b\in\mathbb{Z}^m$ such that $\phi(v)=vM+b$ holds for every $v\in\{0,1\}^m$. We claim that each column of $A$ has at most one non zero element. Assume by contradiction that both $M_{i,j}$ and $M_{k,j}$ are non zero for some $i\neq k$. Let $v_1,v_2,v_3,v_4$ be the $4$ possible vectors in $\{0,1\}^n$ with zeroes outside of the coordinates $i$ and $k$. We have that the $j$-th coordinates of $v_sM, s\in [4]$ are in some order $b_j,b_j+M_{i,j},b_j+M_{k,j},b_j+M_{i,j}+M_{k,j}$ among which there are at least $3$ different numbers. This contradicts the fact that $vM$ is $0-1$ vector for every $v\in\{0,1\}^n$. We obtained that for every $j\in [m]$ there is some $j'\in [n]$ such that $p^m_j\circ\phi$ is a function of $v_{j'}$. 
\end{proof}

For the following proposition let $F_k=\langle g_1,g_2\dots,g_k\rangle$ denote the free group in $k$ generators and let $$F_k^\square:=\{\prod_{i=1}^k g_i^{v_i}~:~v\in\{0,1\}^k\}\subset F_k.$$ Note that we use the convention that $\prod_{i=1}^k a_i$ denotes the "left to right" product $a_1a_2a_3\dots a_k$ in a general group. 

\begin{proposition}\label{homeq}  A map $\phi:\{0,1\}^n\rightarrow\{0,1\}^m$ is a morphism in $\mathcal{G}$ if and only if the map $f:F_n^\square\to F_m^\square$ defined by 
\begin{equation}\label{homeqeq}
f(\prod_{i=1}^n g_i^{v_i})=\prod_{i=1}^m g_i^{\phi(v)_i}
\end{equation}
extends to an affine morphism from $F_n$ to $F_m$. 
\end{proposition}

\begin{proof} First assume that $\phi$ is a morphism in $\mathcal{G}$. Let $f':\{0,1\}^n\to F_m^\square$ be the function defined by $f'(v):=\prod_{i=1}^m g_i^{\phi(v)_i}$.  Let us partition $[m]$ into $n$ (possibly empty) consecutive intervals $P_1,P_2,\dots,P_n$ such that for $i\in [n]$ and $a\in P_i$ we have $\tau_\phi(a)\in\{0,i\}$ (we use the function $\tau_\phi$ introduced in definition \ref{cubmorph}). The existence of such a partition follows from the monotonicity of $\tau_\phi$ on its support. For $i\in [n]$ let $f'_i(v):=\prod_{j\in P_j} g_j^{\phi(v)_j}$ (empty products are defined to be equal to the identity element). Observe that for each $i\in [n]$ the function $f'_i(v)$ depends only on $v_i$ and so it can take two different values, say $a_{i,0}\in F_m$ if $v_i=0$ and $a_{i,1}\in F_m$ if $v_i=1$. Now we have that 
$$f'(v)=\prod_{i=1}^n f'_i(v)=\prod_{i=1}^n a_{i,v_i}.$$ By lemma \ref{homeqlem} there are elements $c,h_1,h_2,\dots,h_n\in F_m$ such that $f'(v)=c\prod_{i=1}^n h_i^{v_i}$. Let $r:F_n\rightarrow F_m$ denote the unique group homomorphism with the property that $r(g_i)=h_i$ holds for $i\in [n]$. We have for $v\in\{0,1\}^n$ that $$f(\prod g_i^{v_i})=f'(v)=c\prod_{i=1}^n h_i^{v_i}=cr(\prod g_i^{v_i})$$ and thus $f$ extends to the affine morphism $x\mapsto cr(x)$.

For the other direction assume that $f:F_n^\square\to F_m^\square$ extends to an affine homomorphism $r:F_n\to F_m$.  First we show that $\phi$ is a morphism in $\mathcal{N}$. To see this let $q$ denote the map from $F_n/[F_n,F_n]=\mathbb{Z}^n$ to $F_m/[F_m,F_m]=\mathbb{Z}^m$ defined by $q(x[F_n,F_n]):=r(x)[F_m,F_m]$. It is clear that $q$ is well defined and that it is an affine morphism. Furthermore we have that the restriction of $q$ to $\{0,1\}^n\subset\mathbb{Z}^n$ is equal to $\phi$. It follows from proposition \ref{abhomeq} that $\phi$ is a morphism in $\mathcal{N}$. It remains to show the monotonicity property for $\tau_\phi$. Assume by contradiction that for some $1\leq i<j\leq m$ we have that $a:=\tau_\phi(j)\neq 0$ and that $b:=\tau_\phi(i)>\tau_\phi(j)$. Let $q_2:F_m\to F_2$ be the unique  homomorphism with the property that $q_2(g_i)=g_1,q_2(g_j)=g_2$ and $q_2(g_k)=1$ if $k\in [m]\setminus\{i,j\}$. We have that $r_2:=q_2\circ r$ is an affine homomorphism from $F_n$ to $F_2$ with the property that $r_2(\prod_{i=1}^n g_i^{v_i})=g_1^{\phi(v)_i}g_2^{\phi(v)_j}$ holds for every $v\in\{0,1\}^n$. In particular, if we set all coordinates in $[n]\setminus\{a,b\}$ to $0$ we have that $r_2(g_a^{\epsilon_1}g_b^{\epsilon_2})=g_1^{t_1(\epsilon_2)}g_2^{t_2(\epsilon_1)}$ holds where $\epsilon_i\in\{0,1\}$ for $i=1,2$ and $t_1,t_2:\{0,1\}\to\{0,1\}$ are non constant functions. By lemma \ref{charaff}, the property that $r_2$ is an affine homomorphism implies that $r_2(1)r_2(g_a)^{-1}r_2(g_ag_b)r(g_b)^{-1}=1$ and thus $$g_1^{t_1(0)}g_2^{t_2(0)}(g_1^{t_1(0)}g_2^{t_2(1)})^{-1}g_1^{t_1(1)}g_2^{t_2(1)}(g_1^{t_1(1)}g_2^{t_2(0)})^{-1}=1.$$ This is equivalent with $$g_1^{t_1(0)}g_2^{t_2(0)-t_2(1)}g_1^{t_1(1)-t_1(0)}g_2^{t_2(1)-t_2(0)}g_1^{-t_1(1)}=1.$$ Since $t_1$ and $t_2$ are both non constant functions, this is impossible in the free group $F_2$.
\end{proof}

\section{nilspaces and groupspaces}\label{nil+group}

Let $\{0,1\}^n_*:=\{0,1\}^n\setminus\{1^n\}$.

\begin{definition} A cube-structure is a pair $\{X,\{C^k(X)\}_{n=0}^\infty\}$ where $X$ is a set and $C^n(X)$ is a subset of $X^{\{0,1\}^n}$.
\end{definition}

Note that we can also think think of the elements in $C^n(X)$ as functions from $\{0,1\}^n$ to $X$. By slightly abusing the notation we will identify cubic structures by their ground set $X$.

\begin{definition}\label{maindef} A cube-structure $X$ is called a {\bf nilspace} if the following axioms hold with $\mathcal{C}=\mathcal{N}$ and it is a {\bf groupspace} if the axioms hold with $\mathcal{C}=\mathcal{G}$. 
\begin{enumerate}
\item {\bf (presheaf)} If $\phi\in \hom_{\mathcal{C}}(\{0,1\}^n,\{0,1\}^m)$ and $\psi\in C^m(X)$ then $\psi\circ\phi\in C^n(X)$.
\item {\bf (ergodicity)} $C^1(X)=X^{\{0,1\}}$.
\item {\bf (completion)} Let $f:\{0,1\}^n_*\to X$ such that for every $i\in [n]$ we have that $f\circ e^n_{i,0}\in C^{n-1}(X)$. Then there is a map $\phi\in C^n(X)$ such that $\phi$ restricted to $\{0,1\}^n_*$ is equal to $f$.
\end{enumerate}
\end{definition}

Note that since the category $\mathcal{G}$ has strictly fewer morphisms than $\mathcal{N}$ we have that the first axiom is a weaker restriction on $(X,\{C^n(X)\}_{n=0}^\infty)$ and so groupspaces are more general than nilspaces. The goal of this paper is to show that the theory of groupspaces is general enough to include arbitrary groups  but it is not too general in the sense that many statements in nilspace theory have natural generalizations to groupspaces. 

\begin{definition} A groupspace $X$ is called a {\bf $k$-step groupspace} if in the third axiom the function $\phi$ is always uniquely determined by $f$ for $n=k+1$. 
\end{definition}

The next lemmas shows that larger dimensional cubes determine smaller dimensional cubes in a groupspace. 

\begin{lemma}\label{down-det} Let $X$ be a groupspace and let $n,k\in\mathbb{N}$ such that $n\leq k$. Then $c\in C^n(X)$ if and only if $c\circ\phi\in C^k(X)$ holds for the function $$\phi(x_1,x_2,\dots,x_k)=(x_1,x_2,\dots,x_n).$$
\end{lemma}

\begin{proof} Since $\phi\in \hom_{\mathcal{G}}(\{0,1\}^k,\{0,1\}^n)$, the presheaf axiom shows that if $c\in C^n(X)$ then so is $c\circ\phi$. For the other direction let $$\psi(x_1,x_2,\dots,x_n):=(x_1,x_2,\dots,x_n,0,0,\dots,0)$$ with $k-n$ zeros. We have that $\phi\in \hom_{\mathcal{G}}(\{0,1\}^n,\{0,1\}^k)$. Since $c=(c\circ\phi)\circ\psi$ we have that if $c\circ\phi\in C^k(X)$ then $c\in C^n(X)$. 
\end{proof}

In the next two lemmas we prove that in a $k$-step groupspace, the set $C^{k+1}(X)$ determins the whole groupspace structure. 

\begin{lemma}\label{up-det} Assume that $X$ is a $k$-step groupspace and $n\geq k+1$. Let $c:\{0,1\}^n\to X$ be a function. Then $c\in C^n(X)$ if and only if $c\circ e^n_{i,0}\in C^{n-1}(X)$ holds for every $i\in[n]$ and furthermore $c\circ\phi\in C^{k+1}(X)$ holds for the map defined by $$\phi(x_1,x_2,\dots,x_{k+1})=(x_1,x_2,\dots,x_{k+1},1,1,\dots,1)$$ with $n-(k+1)$ ones. 
\end{lemma} 

\begin{proof} Let $c'$ be the restriction of $c$ to $\{0,1\}^n_*$.  If $c\in C^n(X)$ then the conditions of the lemma are satisfied by the presheaf axiom. On the other hand if the conditions of the lemma are satisfied then by the completion axiom there is $c_2\in C^n(X)$ such that its restriction to $\{0,1\}^n_*$ is equal to $c'$. The $k$-step property of $X$ implies that $c_2\circ\phi=c\circ\phi$ since both are in $C^{k+1}(X)$ and they agree on $\{0,1\}^{k+1}_*$. This means that $c_2=c$ and thus $c\in C^n(X)$.
\end{proof}

\begin{lemma}\label{k-step-det} Assume that $X$ has two parallel $k$-step groupspace structures $\{C^n(X)\}_{n=0}^\infty$ and $\{C'^n(X)\}_{n=0}^\infty$. Assume furthermore that $C^{k+1}(X)=C'^{k+1}(X)$. Then $C^n(X)=C'^n(X)$ holds for every $n$.
\end{lemma}

\begin{proof} Lemma \ref{down-det} shows that if $n<k+1$ then the statement follows. Assume that $n>k+1$ and also assume that by induction we have the statement for $n-1$. Then lemma \ref{up-det} shows that $C^n(X)=C'^n(X)$.
\end{proof}

\begin{definition} Let $X,Y$ be two groupspaces. We say that a map $m:X\to Y$ is a {\bf groupspace morphism} if for every $n\in\mathbb{N}$ and $f\in C^n(X)$ we have that $m\circ f\in C^n(Y)$.  
\end{definition}

Note that the above definition of morphism does not differentiate between $\mathcal{N}$ and $\mathcal{G}$ and so groupspace morphisms between nilspaces are nilspace morphisms. This also justifies that nilspaces are just special elements in the category of groupspaces. 

\begin{definition} A {\bf compact groupspace} is a groupspace $(X,\{C^n(X)\}_{n=0}^\infty)$ such that $X$ is a compact  topological space and $C^n(X)$ is a closed subset of $X^{\{0,1\}^n}$ for every $n\in\mathbb{N}$.
\end{definition}

\section{Classification of $1$-step groupspaces}\label{Class-1}

In this chapter we show that non commutative groups are equipped with a natural $1$-step cubespace structure. We also show the converse, namely that every $1$-step cubespace comes from a group with this construction. 

\begin{definition} Let $G$ be an arbitrary group. We denote by $\mathcal{D}_1(G)$ the cube-structure whose ground set is $G$ and for $k\in\mathbb{N}$ we have that $C^k(G)$ is the set of all maps of the form (\ref{simplecubes}) in chapter \ref{cubesin}.
\end{definition}

From this definition we have that $C^k(G)$ is the set of all maps of the form $f\circ\tau_k$ where $f:F_k\to G$ runs through all affine homomorphisms. We will need the next lemmas.

\begin{lemma}\label{lowcube}  Let $G$ be a group together with the cube-structure $\mathcal{D}_1(G)$. Then $C^0(G)=G,C^1(G)=G\times G$ and 
$$C^2(G)=\{f:\{0,1\}^2\to G : f(0,0)f(1,0)^{-1}f(1,1)f(0,1)^{-1}=1\}$$
\end{lemma}

\begin{proof} The first two claims are trivial from (\ref{simplecubes}). To see the last claim assume first that $f\in C^2(G)$ and thus 
\begin{equation}\label{1classeq1}
f(0,0)=g_0,f(1,0)=g_0g_1,f(0,1)=g_0g_2,f(1,1)=g_0g_1g_2.
\end{equation}
Then we have that $$f(0,0)f(1,0)^{-1}f(1,1)f(0,1)^{-1}=g_0(g_1^{-1}g_0^{-1})(g_0g_1g_2)(g_2^{-1}g_0^{-1})=1.$$ On the other hand if $f$ satisfies the condition of the lemma then $$f(1,1)=f(1,0)f(0,0)^{-1}f(0,1).$$ Let 
$$g_0:=f(0,0),g_1:=f(0,0)^{-1}f(1,0),g_2:=f(0,0)^{-1}f(0,1).$$ With this notation we have that the first three of the equations in (\ref{1classeq1}) hold. To see the last one observe that
$$g_0g_1g_2=f(0,0)f(0,0)^{-1}f(1,0)f(0,0)^{-1}f(0,1)=f(1,0)f(0,0)^{-1}f(0,1)=f(1,1).$$
\end{proof}

For the next lemma let $\{0,1\}^k_1$ denote the subset in $\{0,1\}^k$ with at most one coordinate with a $1$. 

\begin{lemma}\label{eqon1} Let $G$ be a group together with the cube-structure $\mathcal{D}_1(G)$.  Let $k\in\mathbb{N}$ and assume that $c_1,c_2\in C^k(G)$ satisfy $c_1(v)=c_2(v)$ for every $v\in\{0,1\}^k_1$. Then $c_1=c_2$.
\end{lemma}

\begin{proof} Formula (\ref{simplecubes}) shows that if $c\in C^k(G)$ then the group elements $g_0,g_1,\dots,g_k$ describing $c$ are determined by the values of $c$ on $\{0,1\}^k_1$ and thus $c$ is uniquely determined by its restriction to $\{0,1\}^k_1$. This completes the proof.
\end{proof}

\begin{theorem} 
Let $G$ be an arbitrary group. Then the cube-structure $\mathcal{D}_1(G)$ is a $1$-step cubespace. 
\end{theorem}

\begin{proof}
We start by checking the cubespace axioms. We start by the presheaf axiom. Assume that $\phi\in\hom_{\mathcal{G}}(\{0,1\}^n,\{0,1\}^m)$ and $\psi\in C^m(G)$. By lemma \ref{simpaffdef} we can write $\psi=q\circ\tau_m$ where $q$ is an affine homomorphism from $F_m$ to $G$. Furthermore by proposition \ref{homeq} we have that there exists an affine morphism $\alpha:F_n\to F_m$ such that $\tau_m^{-1}\circ\alpha\circ\tau_n=\phi$.  It follows that $\psi\circ\phi=q\circ\alpha\circ\tau_n$ where $q\circ\alpha$ is an affine morphism from $F_n$ to $G$. This $\psi\circ\phi\in C^n(G)$ which verifies the presheaf axiom. 

The ergodicity axiom follows from lemma \ref{lowcube}.

Now we turn to the completion axiom. Lemma \ref{lowcube} implies the completion axiom for $n=0,1$. Assume that $n>1$. Let $f:\{0,1\}^n_*\to X$ with the property in the axiom. Let $q:F_n\to G$ denote the unique group homomorphism such that $q(g_i)=f(0)^{-1}f(b_i)$ where $b_i\in\{0,1\}^n$ has $1$ at the $i$-th coordinate and $0$ everywhere else. We have that $q':=f(0)q$ is an affine morphism from $F_n$ to $G$ such that $q'\circ\tau_n$ agrees with $f$ on every vector in $\{0,1\}^n_1$. Notice that $\phi:=q'\circ\tau_n$ is in $C^n(G)$ holds by definition. We claim that $\phi$ satisfies the requirement in the closing axiom. We need to show that $\phi=f$ holds on $\{0,1\}^n_*$. Let us use the short hand notation $e_i$ for the map $e^n_{i,0}$. Notice that $\{0,1\}^n_*$ is the union of the image sets of the maps $e_i:\{0,1\}^{n-1}\to\{0,1\}^n$ and thus it suffices to show that $\phi\circ e_i=f\circ e_i$ holds for $1\leq i\leq n$. Observe that $\phi\circ e_i$ agrees with $f\circ e_i$ on $\{0,1\}^{n-1}_1$.  Lemma \ref{eqon1} implies that $\phi\circ e_i=f\circ e_i$ and thus the proof of the completion axiom is finished. Lemma \ref{lowcube} implies that if $f\in C^2(G)$ then $f(1,1)=f(1,0)f(0,0)^{-1}f(0,1)$. Thus completion must be unique for $n=2$ and so $G$ is $1$-step. 
\end{proof}

\begin{theorem}\label{class1step} Every $1$-step groupspace is isomorphic with $\mathcal{D}_1(G)$ for some group $G$. 
\end{theorem}

\begin{proof} Assume that $X$ is a $1$-step groupspace. We introduce a binary and a unary operation $b:X\times X\to X$ and $u:X\to X$ in the following way. Let us fix an element $e\in X$. Let $u(x)$ be the unique completion of the corner $f(0,0)=x,f(1,0)=e,f(0,1)=e$ and let $b(x,y)$ be the unique completion of the corner $f(0,0)=e,f(1,0)=x,f(0,1)=y$. We claim that $(X,b,u,e)$ is a group where $b$ is the multiplication and $u$ is the inverse map end $e$ is the identity element. 

Let us first prove the associativity: $b(x,b(y,z))=b(b(x,y),z)$. To see this we define the corner map $f:\{0,1\}^3_*\to X$ by $f(0,0,0)=e,f(1,0,0)=x,f(0,1,0)=y,f(0,0,1)=z,f(1,1,0)=b(x,y),f(1,0,1)=b(x,z),f(0,1,1)=b(y,z)$. Since $f$ satisfies the requirement of the completion axiom we have that it can be extended to a map $c\in C^3(X)$ which agrees with $f$ on $\{0,1\}^3_*$. Now let $\phi_1,\phi_2:\{0,1\}^2\to\{0,1\}^3$ be defined by $\phi_1(x_1,x_2)=(x_1,x_1,x_2)$ and $\phi_2(x_1,x_2)=(x_1,x_2,x_2)$. Both maps are morphisms in $\mathcal{G}$ and so we have that both $c\circ \phi_1$ and $c\circ\phi_2$ are in $C^2(X)$. On the other hand by unique completion we have that $$b(b(x,y),z)=f\circ \phi_1(1,1)=c(1,1,1)=f\circ \phi_2(1,1)=b(x,b(y,z)).$$ 

Now we prove that $b(e,x)=x$. Let $\phi:\{0,1\}^2\to\{0,1\}$ be the projection $\phi(x_1,x_2)=x_2$. Since $\phi\in\mathcal{G}$ we by the ergodicity axiom that $f\circ\phi\in C^2(X)$ where $f(0)=e, f(1)=x$. Notice that $f\circ\phi(0,0)=e,f\circ\phi(1,0)=e,f\circ\phi(0,1)=x,f\circ\phi(1,1)=x=b(e,x)$.

We claim also that $b(u(x),x)=e$. Notice that by definition of $u$, the map $f:\{0,1\}^2\to X$ defined by $f(0,0)=x,f(0,1)=e,f(1,0)=e,f(1,1)=u(x)$ is in $C^2(X)$. Let $\phi(x_1,x_2):=(x_1,1-x_2)$. Then $f\circ\phi\in C^2(X)$. On the other hand $f\circ\phi(0,0)=e,f\circ\phi(1,0)=u(x),f\circ\phi(0,1)=x,f\circ\phi(1,1)=e=b(u(x),x)$.

Recall that right inverse and right identity is enough for the group axioms and thus we finished the proof of the claim that $(X,b,u,e)$ is a group. From now on we will use normal group operations (product and inverse) to replaece $b$ and $u$. It remains to show that $\mathcal{D}_1(X,b,u,e)$ has the same same cubes as $X$. To see this let $x,y,z\in X$ be arbitrary elements and let $f:\{0,1\}^3\to X$ be defined by $f(1,0,0)=x,f(1,1,0)=y,f(1,0,1)=z,f(1,1,1)=yx^{-1}z,f(0,0,0)=e,f(0,1,0)=x^{-1}y,f(0,0,1)=x^{-1}z,f(0,1,1)=x^{-1}yx^{-1}z$. The values of $f$ are chosen so that by the unique completion rules it has to be in $C^3(X)$. On the other hand the composition of $f$ with the map $\phi(x_1,x_2)=(1,x_1,x_2)$ shows that the unique completion of the corner map defined by $g(0,0)=x,g(1,0)=y,g(0,1)=z$ is $yx^{-1}z$. Thus $C^2(X)=C^2(\mathcal{D}_1(X,b,u,e))$ holds. By lemma \ref{k-step-det}, in a $1$-step groupspace $C_2(X)$ determines $C^k(X)$ for every $k$ and thus $C^k(X)=C^k(\mathcal{D}_k(X,b,u,e))$ holds.
\end{proof}

\section{Abelian groups with higher degree structures}\label{highdegabgroups}

Abelian groups with higher degree structures were introduced in \cite{Szegnil} and they play a crucial role in nilspace theory. These structures will also be crucial in understanding groupspaces in general.
We will need the following notation.

\begin{definition} Let $n\in\mathbb{N}$ , $A$ be an abelian group and let $f:\{0,1\}^n\to A$ be an arbitrary function. then 
$$w(f):=\sum_{v\in\{0,1\}^{n}}(-1)^{h(v)}f(v).$$
\end{definition}

The next theorem introduces higher degree structures on abelian groups. 

\begin{theorem}[Higher degree structures on abelian groups]\label{hdsog} Let $A$ be an arbitrary abelian group and let $k\geq 1$ be a natural number.  Then there exists a unique nilspace structure $\mathcal{D}_k(A)$ on the set $A$ with the property that $c\in C^{k+1}(\mathcal{D}_k(A))$ if and only if $w(c)=0$.
\end{theorem}

\begin{remark}\label{polyrem} It is interesting to mention that continuous morphisms from $\mathcal{D}_1(\mathbb{R})$ to $\mathcal{D}_k(\mathbb{R})$ are the polynomials in $\mathbb{R}[x]$ of degree at most $k$.
\end{remark} 

Note that it follows from theorem \ref{hdsog} that $\mathcal{D}_k(A)$ is a $k$-step nilspace and furthermore $C^i(\mathcal{D}_k(A))=A^{\{0,1\}^i}$ holds for every $i\leq k$. This latter property is a strengthening of the ergodicity axiom and called $k$-ergodic. In the rest of this chapter we prove this theorem.

\begin{lemma}\label{lhd1} Let $\phi:\{0,1\}^n\to\{0,1\}^m$ be a map in $\hom_\mathcal{N}(\{0,1\}^n,\{0,1\}^m)$. Then either $\phi$ is injective or there exists $i\in [n]$ such that $\phi=\phi\circ r_i$ holds where $r_i$ is the reflection introduced in Chapter \ref{cubmor}.
\end{lemma}

\begin{proof} If $\phi$ is not injective then there exists $i\in[n]$ such that $\gamma_\phi(j)\neq i$ holds for every $j\in [m]$. It is clear that in this case the value of $\phi$ can not depend on the $i$-th coordinate and thus the claim of the lemma holds. 
\end{proof} 

\begin{lemma}\label{lhd2} Let $c:\{0,1\}^n\to A$ be a function such that $c=c\circ r_i$ holds for some $i\in [n]$. Then $w(c)=0$.
\end{lemma}

\begin{proof} To see this, notice that $c(v)$ and $c(r_i(v))$ are counted with opposite signs in the definition of $w(c)$ but they have the same values and thus they cancel each other for every $v\in\{0,1\}^n$.
\end{proof} 

\begin{lemma}\label{lhd3} Let $f:\{0,1\}^n\to A$ be arbitrary, $m\in\mathbb{N}$ and assume that $\phi\in\hom_\mathcal{N}(\{0,1\}^m,\{0,1\}^n)$ is not injective. Then $w(f\circ\phi)=0$.
\end{lemma}

\begin{proof} The proof is a direct consequence of lemma \ref{lhd1} and lemma \ref{lhd2}.
\end{proof}



\begin{definition}\label{cspaces} Let $k\geq 1,n$ be natural numbers and $A$ be an abelian group. Then we denote by $V(A,n,k)\subset A^{\{0,1\}^n}$ the subgroup generated by all functions of the form $f\circ\phi$ where $m\leq k$, $f:\{0,1\}^m\to A$ is an arbitrary function and $\phi\in\hom_\mathcal{N}(\{0,1\}^n,\{0,1\}^m)$. We denote by $V'(A,n,k)$ the set of all functions $f:\{0,1\}^n\to A$ such that $w(f\circ s_v)=0$ holds for every $v\in\{0,1\}^n$ with $h(v)\geq k+1$.
\end{definition}

Notice that $f\mapsto w(f\circ s_v)$ is a group homomorphism from $A^{\{0,1\}^n}$ to $A$ and thus its kernel is a subgroup. It follows that $V'(A,n,k)$ is an intersection of subgroups and thus it is a subgroup in $A^{\{0,1\}^n}$.

\begin{lemma}\label{lhd4} Let $k\geq 1,n$ be natural numbers and $A$ be an abelian group. Then  $V(A,n,k)\subseteq V'(A,n,k)$. 
\end{lemma}

\begin{proof} Since $V'(A,n,k)$ is a subgroup in $A^{\{0,1\}^n}$ it is enough to show that all the generators of $V(A,n,k)$ are contained in it. To see this let $m\leq k$, $f:\{0,1\}^m\to A$ be a function, let $\phi\in\hom_\mathcal{N}(\{0,1\}^n,\{0,1\}^m)$ and let $v\in\{0,1\}^n$ be such that $h(v)\geq k+1$. Then $\phi\circ s_v\in\hom_\mathcal{N}(\{0,1\}^{h(v)},\{0,1\}^m)$ and thus it can not be injective since $m<h(v)$. It follows from lemma \ref{lhd3} that $w(f\circ\phi\circ s_v)=0$ and thus $f\circ\phi\in V'(A,n,k)$.
\end{proof}

\begin{lemma}\label{lhd5} Let $k\geq 1,n$ be natural numbers and $A$ be an abelian group. If $f,g\in V'(A,n,k)$ and $f(v)=g(v)$ holds for every $v\in\{0,1\}^n$ with $h(v)\leq k$ then $f=g$.
\end{lemma}

\begin{proof} Assume by contradiction that $f\neq g$. Let $v\in\{0,1\}^n$ be a vector of minimal hight $h(v)$ such that $f(v)\neq g(v)$. We have that $h(v)\geq k+1$. Notice that $f\circ s_v$ and $g\circ s_v$ differ only at $1^{h(v)}$ and thus $w(f\circ s_v)\neq w(g\circ s_v)$. On the other hand $f,g\in V'(A,n,k)$ implies that $w(f\circ s_v)=w(g\circ s_v)=0$, a contradiction.
\end{proof}

For $n,m\in\mathbb{N}$ let $\{0,1\}^n_m$ denote the set of vectors $v$ in $\{0,1\}^n$ with $h(v)\leq m$. 

\begin{lemma}\label{lhd6} Let $k\geq 1,n$ be natural numbers and $A$ be an abelian group. Let $f:\{0,1\}^n_k\to A$ be an arbitrary function. Then there exists $g\in V(A,n,k)$ such that its restriction to $\{0,1\}^n_k$ equals to $f$.
\end{lemma}

\begin{proof} For $v\in\{0,1\}^n$ and $a\in A$ let $h_{v,a}:=f\circ p_{\supp{v}}$ where $f:\{0,1\}^{h(v)}\to A$ is the function which takes the value $a$ at $1^{h(v)}$ and $0$ everywhere elese. We have that if $h(v)\leq k$ then $h_{v,a}\in V(A,n,k)$. We have that $h_{v,a}(w)=0$ if $h(w)<h(v)$ or if $h(w)=h(v)$ and $w\neq v$. We also have that $h_{v,a}(v)=a$. We define a sequence of functions $g_0,g_1,\dots,g_k:\{0,1\}^n\to A$ recursively. We define $g_0$ to be the constant $f(0^n)$ function. It is clear that $g_0\in V(A,n,k)$. Assume that $g_i$ is defined for some $i<k$. Then we set $$g_{i+1}:=g_i+\sum_{v\in\{0,1\}^n,h(v)=i+1}h_{v,f(v)-g_i(v)}.$$ It is clear from this definition that $g_i\in V(A,n,k)$ holds for $0\leq i\leq k$ and that $g_i$ agrees with $f$ on $\{0,1\}^n_i$. Then $g:=g_k$ satisfies the requirement of the lemma. 
\end{proof}

\begin{proposition}\label{phd6} Let $k\geq 1,n$ be natural numbers and $A$ be an abelian group. Then $V'(A,n,k)=V(A,n,k)$. 
\end{proposition}

\begin{proof} According to lemma \ref{lhd4} it is enough to show that $V'(A,n,k)\subseteq V(A,n,k)$. Let $f\in V'(A,n,k)$. According to lemma \ref{lhd6} there exists $g\in V(A,n,k)$ such that $g$ is equal to $f$ on the set $\{0,1\}^n_k$. It follows from lemma \ref{lhd4} and lemma \ref{lhd5} that $f=g$ and thus $f\in V(A,n,k)$.
\end{proof}

\begin{definition} Let $A$ be an abelian group and $k\in\mathbb{N}$ with $k\geq 1$. Then we define the cubic structure $\mathcal{D}_k(A)$ by $C^n(\mathcal{D}_k(A)):=V(A,n,k)=V'(A,n,k)$ for $n\in\mathbb{N}$.
\end{definition}

Notice that $C^{k+1}(\mathcal{D}_k(A))=V'(A,k+1,k)$ consists of all maps $c:\{0,1\}^{k+1}\to A$ with $w(c)=0$. Thus this cubic structure satisfies theorem \ref{hdsog}. 

\begin{proposition}\label{phd7} Let $k\geq 1$ and $A$ be an abelian group. Then $A$ with the cubic structure $\mathcal{D}_k(A)$ is a nilspace. 
\end{proposition}

\begin{proof} We have from the definitions that $C^i(\mathcal{D}_k(A))=A^{\{0,1\}^i}$ holds for $i\leq k$ and thus the ergodicity axiom holds in an even stronger form. The see the presheaf axiom let $n,i\in\mathbb{N}$, $c\in C^n(\mathcal{D}_k(A))$ and assume that $\phi\in\hom_\mathcal{N}(\{0,1\}^i,\{0,1\}^n)$. We show that $c\circ\phi\in V(A,i,k)$. Since $c\in V(A,n,k)$ we have that it is generated (in the group $A^{\{0,1\}^n}$) by functions of the form $f\circ\psi$ where $f:\{0,1\}^m\to A$ is an arbitrary function and $\psi\in\hom_\mathcal{N}(\{0,1\}^n,\{0,1\}^m)$ for $m\leq k$. Using linearity, we obtain that $c\circ\phi$ is generated (in the group $A^{\{0,1\}^i}$) by functions of the form $(f\circ\psi)\circ\phi$ where $f$ is as above. Since $(f\circ\psi)\circ\phi=f\circ(\psi\circ\phi)$ and $\psi\circ\phi$ is a morphism we obtain that $(f\circ\psi)\circ\phi\in V(A,i,k)$. This finishes the proof of the presheaf axiom. 

To see the completion axiom let $f:\{0,1\}^n_*\to A$ be a function satisfying the condition of definition \ref{maindef}. It is clear from this condition that for every $v\in\{0,1\}^n_*$ with $h(v)>k$ we have that $w(f\circ s_v)=0$. This follows from the fact that if $v_i=0$ holds for some $i\in [n]$ then there exists $v'$ with $e^n_{i,0}(v')=v$. Then $w(f\circ s_v)=w(f\circ e^n_{i,0}\circ s_{v'})=0$. Let $c:\{0,1\}^n\to A$ be the function uniquely determined by the property that $w(c)=0$ and $c(v)=f(v)$ for $v\in\{0,1\}^n_*$. It is clear that $c\in V'(A,n,k)$ and thus the completion axiom holds. 
\end{proof}

Theorem \ref{hdsog} follows directly from proposition \ref{phd7} and lemma \ref{up-det} by noticing that $\mathcal{D}_k(A)$ is $k$-step with this definition.

\section{Examples for higher step groupspaces}\label{exhigh}

All nilspaces are groupspaces and they provide examples for complicated families of $k$-step groupspaces for arbitrary $k$. In this chapter we give a family of examples of higher step groupspaces that are not nilspaces. Let $G$ be a non-commutative group with a non-trivial central subgroup $Z<G$ and let us fix a natural number $k$. We define a $k$-step groupspace structure $\mathcal{H}_{Z,k}(G)$ on $G$ by adding new cubes to $\mathcal{D}_1(G)$. We say that $c:\{0,1\}^n\to G$ is a cube in $\mathcal{H}_{Z,k}(G)$ if there exists a cube $f\in C^n(\mathcal{D}_k(Z))$ and $g\in C^n(\mathcal{D}_1(G))$ such that $c(x)=f(x)g(x)$ holds for every $x\in\{0,1\}^n$. We will prove the next theorem.

\begin{theorem}\label{groupex} The cube structure $\mathcal{H}_{Z,k}(G)$ is a $k$-step groupspace.
\end{theorem}

Let $h:G\to G/Z$ denote the usual factor map. We will need the next lemma.

\begin{lemma}\label{rescube} Let $n\in\mathbb{N}$ and let $c_1,c_2\in C^n(\mathcal{D}_1(G))$ such that $h\circ c_1=h\circ c_2$. then the map $f:\{0,1\}^n\to Z$ defined by $f(x)=c_1(x)^{-1}c_2(x)$ is in $C^n(\mathcal{D}_1(Z))$.
\end{lemma}

\begin{proof} The map $f$ is in $\mathcal{D}_1(Z)$ if and only if $f\circ\phi\in C^2(\mathcal{D}_1(Z))$ holds for an arbitrary cube morphism $\phi:\{0,1\}^2\to\{0,1\}^n$ in $\mathcal{G}$. This reduces the problem to the case where $n=2$. In this case we have that
$$1=c_2(0,0)c_2(1,0)^{-1}c_2(1,1)c_2(0,1)^{-1}=$$
$$=c_1(0,0)f(0,0)f(1,0)^{-1}c_1(1,0)^{-1}c_1(1,1)f(1,1)f(0,1)^{-1}c_1(0,1)^{-1}=$$
$$=c_1(0,0)c_1(1,0)^{-1}c_1(1,1)c_1(0,1)^{-1}f(0,0)f(1,0)^{-1}f(1,1)f(0,1)^{-1}=$$
$$=f(0,0)f(1,0)^{-1}f(1,1)f(0,1)^{-1}$$
where the first and the fourth equality is by lemma \ref{lowcube}, the second one uses the fact $c_2(x)=c_1(x)f(x)$ and the third uses that $f$ takes values in the center of $G$ and thus they commute with everything. Lemma \ref{lowcube} shows that $f\in C^2(\mathcal{D}_1(Z))$ holds.
\end{proof}

Now we are ready to prove theorem \ref{groupex}.

\begin{proof} The first two groupspace axioms are trivial from the definitions. We check the completion axiom. Let $f:\{0,1\}^n_*\to G$ be a corner in $\mathcal{H}_{Z,k}(G)$ for some natural number $n>1$. For $i\in [n]$ we have that $f\circ e^n_{i,0}\in C^{n-1}(\mathcal{H}_{Z,k}(G))$ and thus it can be written as $c_ig_i$ where $c_i\in C^{n-1}(\mathcal{D}_1(G))$ and $g_i\in C^{n-1}(\mathcal{D}_k(Z))$. 
Let $c$ be the unique cube in $C^n(\mathcal{D}_1(G))$ which agrees with $f$ on $\{0,1\}^n_1$. We have that $h\circ f=h\circ c$ holds. It follows by lemma \ref{rescube} that the function $$x\mapsto (c\circ e^n_{i,0})(x)^{-1}(f\circ e^n_{i,0})(x)=(c\circ e^n_{i,0})(x)^{-1}c_i(x)g_i(x)$$ is a cube in $\mathcal{D}_k(Z)$. Consequently we have that the function $x\mapsto c(x)^{-1}f(x)$ defined on $\{0,1\}^n_*$ is a corner in $\mathcal{D}_k(Z)$. Thus it has a completion $c'\in C^n(\mathcal{D}_k(Z))$. We have that the function $x\mapsto c(x)c'(x)$ is in $C^n(\mathcal{H}_{Z,k}(G))$ and it is a completion of $f$.
\end{proof}

\section{simplicial gluing and the three-cube construction}

In this chapter we describe the generalization of an important algebraic tool to groupspaces which proved to be very useful in nilspace theory.

\begin{definition} A set $S\subseteq\{0,1\}^n$ is called simplicial if for every vector $v\in S$ and $A\subseteq \supp{v}$ we have $1_A\in S$.  
\end{definition}

It is clear that $S\subseteq\{0,1\}^n$ is simplicial if and only if for every $v\in S$ we have that the image of  $s_v$ is in $S$. A vector $v\in S$ is called maximal if there is no $w\in S$ with $w\neq v$ and $\supp{v}\subset\supp{w}$. Note that if $v\in S$ is maximal then $S\setminus\{v\}$ is again simplicial. 

\begin{definition} We say that a map $f:S\to X$ to a cubespace $X$ is cube preserving if for every $v\in S$ we have $f\circ s_v\in C^{h(v)}(X)$. 
\end{definition}

Not that $f$ is cube preserving if and only if $f\circ s_v\in C^{h(v)}(X)$ holds for every maximal vector $v\in S$.

\begin{lemma}\label{simpglue} Let $X$ be a groupspace and let $S\subseteq\{0,1\}^n$ be a simplicial set for some $n\in\mathbb{N}$. Assume that a map $f:S\to X$ is cube preserving. Then there is $\phi\in C^n(X)$ such that $\phi |_S=f$.
\end{lemma}

\begin{proof} We claim that if $S\subseteq S'$ for some simplicial set $S'$ then there is a cube preserving map $f':S'\to X$ such that $f'|_S=f$. We prove this claim by induction on $|S'|$. This statement holds for $|S'|=|S|$ because then $S'=S$. Assume that it is already true for some number $k$ and let $S'$ be such that $S\subseteq S'$ and $|S'|=k+1$. Let $v\in S'$ be a maximal element.  We have that $S'':=S'\setminus\{v\}$ is simplicial. By our induction hypothesis the statement holds for $S''$ and thus there is a cube preserving map $f'':S''\to X$ with $f''|_S=f$. It follows that $f''\circ s_v\circ e^{h(v)}_{i,0}\in C^{h(v)-1}(X)$ holds for every $i\in [h(v)]$. We get by the completion axiom in $X$ that there is a function $q\in C^{h(v)}(X)$ such that $q$ restricted to $\{0,1\}^{h(v)}_*$ is equal to the composition of $f''$ with the restriction of  $s_v$ to $\{0,1\}^{h(v)}_*$. Let $x:=q(1^{h(v)})$ and let us define $f':S'\to X$ such that $f'|{S''}=f''$ and $f'(v)=x$. We have that $f'\circ s_v=q\in C^{h(v)}(X)$. This show that the function $f'$ is cube preserving and so the proof of the claim is finished.

Now we apply the claim for $S':=\{0,1\}^n$ to obtain a cube preserving function $\phi:\{0,1\}^n\to X$ that extends $f$. Since $1^n\in S'$ we obtain that $\phi=\phi\circ s_{1^n}\in C^n(X)$. 
\end{proof}

Now we turn to the three-cube construction that proved to be a very useful tool in nilspace theory. To generalize it for groupspaces we have to be careful with the ordering of the coordinates. For $n\in\mathbb{N}$ let $T_n\subset \{0,1\}^{2n}$ be the set of vectors $v$ such that $v_{2i-1}+v_{2i}\leq 1$ holds for every $i\in [n]$. It is clear that $T_n$ is a simplicial set. A vector $v\in T_n$ is maximal if and only if  $v_{2i-1}+v_{2i}=1$ holds for every $i\in [n]$. Let $T'_n$ denote the set of maximal elements in $T_n$. 

We define the maps ${\rm fold}_n:T_n\to\{0,1\}^n$ and ${\rm flat}_n:T_n\to\{0,1\}^n$ such that the $i$-th coordinate of ${\rm fold}_n(v)$ is equal to $v_{2i-1}+v_{2i}$ and the $i$-th coordinate of ${\rm flat}_n(v)$ is equal to $v_{2i}$.

We will use the next lemma.

\begin{lemma}\label{foldlem} Let $X$ be a groupspace and $c\in C^n(X)$. Then $c\circ{\rm fold}_n$ and $c\circ{\rm flat}_n$ are both cube preserving on $T_n$.
\end{lemma}

\begin{proof} Let $v\in T'_n$. We have from the definitions that ${\rm fold}_n\circ s_v$ is the identity map on $\{0,1\}^n$. Thus $(c\circ{\rm fold}_n)\circ s_v=c$ and the proof of the first claim is complete. To prove the second claim we observe that if $v\in T'_n$ then ${\rm flat}_n\circ s_v$ is equal to $c\circ\phi_v$ where 
$$\phi_v(x_1,x_2,\dots,x_n)=(v_2x_1,v_4x_2,\dots,v_{2i}x_i,\dots,v_{2n}x_n).$$ It is clear that $\phi_v:\{0,1\}^n\to\{0,1\}^n$ is a morphism in $\mathcal{G}$ and so the proof is complete. 
\end{proof}

Let $\alpha_n\in\hom_{\mathcal{G}}(\{0,1\}^n,\{0,1\}^{2n})$ be the map defined by $\alpha_n(x_1,x_2,\dots,x_n):=(1-x_1,x_1,1-x_2,x_2,\dots,1-x_n,x_n)$. We have that the image of $\alpha_n$ is $T'_n$. The next lemma is an easy consequence of lemma \ref{simpglue}.

\begin{lemma}\label{threelem} Let $n\in\mathbb{N}$ and let $f:T_n\to X$ be a cube preserving map to some cubespace $X$. Then $f\circ\alpha_n\in C^n(X)$.
\end{lemma}

\begin{proof} By lemma \ref{simpglue} we obtain a map $c\in C^n(X)$ such that $c|_{T_n}=f$. We have that $f\circ\alpha_n=c\circ\alpha_n$. From the presheaf axiom we obtain that $c\circ\alpha_n\in C^n(X)$.
\end{proof}

\section{Composition of cubes in groupspaces}

\begin{lemma}\label{facegluing} Let  $X$ be a cube space, $n\in\mathbb{N},i\in[n]$ and assume that $c_1,c_2\in C^n(X)$ are cubes with the property that $c_1\circ e^n_{i,1}=c_2\circ e^n_{i,0}$ holds. Let $c_3:\{0,1\}^n\to X$ be the function determined by the property that $c_3\circ e^n_{i,0}=c_1\circ e^n_{i,0}$ and $c_3\circ e^n_{i,1}=c_2\circ e^n_{i,1}$. Then $c_3\in C^n(X)$ holds. 
\end{lemma}

\begin{proof} 
Let $r_i$ denote the reflection as defined in chapter \ref{cubmor}.  Let $c_1':=c_1\circ r_i$. Let $S:=\{v:v\in\{0,1\}^{n+1},v_iv_{i+1}=0\}$. We have that $S$ is a simplicial set. Let $f:S\to X$ be the map satisfying that $f\circ e^{n+1}_{i+1,0}=c'_1$ and $f\circ e^{n+1}_{i,0}=c_2$. The properties of $c_1$ and $c_2$ guarantee that $f$ is well defined and unique. Lemma \ref{simpglue} implies that there is $f'\in C^{n+1}(X)$ such that $f'_S=f$. Let $\kappa:\{0,1\}^n\to\{0,1\}^{n+1}$ be the cube morphism defined by 
$$\kappa(v_1,v_2,\dots,v_n)=(v_1,v_2,\dots,v_{i-1},1-v_i,v_i,v_{i+1},\dots,v_n).$$ Then we have that $f'\circ\kappa=c_3$ and thus the proof is complete.
\end{proof}

\begin{definition} If $c_1,c_2,c_3$ satisfy the conditions of lemma \ref{facegluing} then we say that $c_1$ and $c_2$ are $i$-{\bf composable} and $c_3$ is the $i$-{\bf composition} of $c_1$ and $c_2$. We denote this relation by $c_3=c_1\boxplus_i c_2$.
\end{definition}

It is clear that if $c_1,c_2,c_3\in C^n(X)$, $c_1,c_2$ are $i$-composable and $c_2,c_3$ are $i$-composable then the pairs $c_1\boxplus_i c_2,c_3$ and $c_1,c_2\boxplus_i c_3$ are both $i$-composable and
\begin{equation}\label{comp-assoc}
(c_1\boxplus_i c_2)\boxplus_i c_3=c_1\boxplus_i(c_2\boxplus_ic_3).
\end{equation}

\begin{definition} Let $X$ be a groupspace and $n\in\mathbb{N},i\in [n+1]$. Let $c_1,c_2\in C^n(X)$. We say that $c_1\thickapprox_i c_2$ if and only if there is $c_3\in C^{n+1}(X)$ such that $c_1=c_3\circ e^{n+1}_{i,0}$ and $c_2=c_3\circ e^{n+1}_{i,1}$.
\end{definition}

The next lemma is an easy consequence of lemma \ref{facegluing}.

\begin{lemma} Let $X$ be a group space and $n\in\mathbb{N},i\in [n+1]$. Then $\thickapprox_i$ is an equivalence relation on $C^n(X)$.
\end{lemma}

\begin{proof} For reflexivity of $\thickapprox_i$ let $c\in C^n(X)$ and let $c':=c\circ p_{[n+1]\setminus\{i\}}$. 
We have that $c=c'\circ e^{n+1}_{i,j}$ holds for $j\in\{0,1\}$ and so $c\thickapprox_i c$. Symmetry follows by composing a reflection $r_i$ of the $i$-th coordinate with $c_3$ from the definition of $\thickapprox_i$. Transitivity follows directly from lemma \ref{facegluing}.
\end{proof}

Another important consequence of lemma \ref{facegluing} is that composition of cubes "respects" the equivalence relation $\thickapprox_i$. More precisely, the equivalence class of a composition depends only on the equivalence classes of the composed cubes.

\begin{lemma}\label{compwell} Let $X$ be a groupspace, $n\in\mathbb{N}$ and $i\in [n+1],j\in [n]$. Let $c_1,c_2,c_1',c_2'\in C^n(X)$ be such that $c_1$ is $j$-composable with $c_2$, $c_1'$ is $j$-composable with $c_2'$, $c_1\thickapprox_i c_1'$ and $c_2\thickapprox_i c_2'$. then $$c_1\boxplus_j c_2\thickapprox_i c_1'\boxplus_j c_2'.$$
\end{lemma}

\begin{proof} For $k=1,2$ let $d_k:\{0,1\}^{n+1}\to X$ denote the function defined by the property that $c_k=d_k\circ e^{n+1}_{i,0}$ and $c_k'=d_k\circ e^{n+1}_{i,1}$. Since $c_k\thickapprox_i c_k'$ we have that $d_k\in C^{n+1}(X)$ holds for $k=1,2$. Let $j'=j$ if $j<i$ and $j'=j+1$ if $j\geq i$. It is easy the check that $d_1$ and $d_2$ are $j'$ composable and that $c_1\boxplus_j c_2=(d_1\boxplus_{j'} d_2) \circ e^{n+1}_{i,0}$ and $c_1'\boxplus_j c_2'=(d_1\boxplus_{j'} d_2) \circ e^{n+1}_{i,1}$. This completes the proof.
\end{proof}

\begin{definition} Let $X$ be a groupspace, $n\in\mathbb{N}$ and $i\in[n+1]$. Let $C^n(X)/\thickapprox_i$ denote the set of $\thickapprox_i$ equivalence classes of $C^n(X)$. We say that two classes $A,B\in C^n(X)/\thickapprox_i$ are $j$-composable for some $j\in [n]$ if there exist $j$-composable cubes $c_1\in A,c_2\in B$. In this case we define $A\boxplus_j B$ as the $\thickapprox_i$ class of $c_1\boxplus_j c_2$. We have by lemma \ref{compwell} that $A\boxplus_j B$ is well defined. 
\end{definition}

\section{Classification of $k$-step, $k$-ergodic groupsapces}\label{Class-k}

\begin{definition} A groupspace $X$ is called $k$-ergodic if $C^k(X)=X^{\{0,1\}^k}$.
\end{definition}

Note that by the axioms every cubespace is at least $1$-ergodic. In chapter \ref{Class-1} we characterized all $1$-step groupspaces. A similar characterization holds for $k$-step groupspaces if we also assume that $X$ is $k$-ergodic. The main difference is that if $k>1$ then such groupspaces come from commutative groups. This fact is reminiscent of the well known theorem that higher homotopy groups are abelian. Even the proof is similar as we use the Eckmann-Hilton argument at some point of the proof. 

It is a basic result in nilspace theory that $k$-step, $k$-ergodic nilspaces are abelian groups equipped with the $k$-th degree structure denoted by $\mathcal{D}_k$ (see Chapter \ref{highdegabgroups}). In the rest of the chapter we generalize this result to groupspaces in the case of $k\geq 2$. 

\begin{theorem}\label{class-k-step-k-ergodic} Let $k\geq 2$ and let $X$ be a $k$-step, $k$-ergodic groupspace. Then $X$ is isomorphic to $\mathcal{D}_k(A)$ for some abelian group $A$.
\end{theorem}

Let $i\in[k+1]$ be fixed for the rest of the proof. Let $Y:=C^k(X)/\thickapprox_i$. Observe that since $X$ is $k$-ergodic we have that $C^k(X)=X^{\{0,1\}^k}$. Let us distinguish an element $x_0\in X$ for the rest of the proof. For $x\in X$ and $w\in\{0,1\}^k$ let $q_{w,x}$ denote the function $q_{v,x}:\{0,1\}^k\to X$ with $q_{w,x}(v)=x$ if $v=w$ and $q_{w,x}(v)=x_0$ if $v\neq w$. Note that $q_{v,x}\in C^k(X)$.

\begin{lemma}\label{gencomp} Let $w\in\{0,1\}^{k+1}$ and $c:\{0,1\}^{k+1}\setminus\{w\}\to X$ be an arbitrary function. Then there is $c_2\in C^{k+1}(X)$ extending $c$ to a cube. 
\end{lemma}

\begin{proof} It follows from lemma \ref{transim} that there is a map $\gamma:\{0,1\}^{k+1}\to\{0,1\}^{k+1}$ generated by reflections such that $\gamma(w)=(1,1,1,\dots,1)$, $\gamma^{-1}=\gamma$ and $\gamma\in\hom_{\mathcal{G}}(\{0,1\}^{k+1},\{0,1\}^{k+1})$. Let $\gamma'$ be the restriction of $\gamma$ to $\{0,1\}^{k+1}_*$. Let $c'=c\circ\gamma'$. The $k$ ergodicity of $X$ implies that $c'$ satisfies the conditions for the completion axiom and so there is an extension $c_2'\in C^{k+1}(X)$. The $c_2:=c_2'\circ\gamma$ satisfies the requirement of the lemma.
\end{proof}

\begin{lemma}\label{eqrep-gen} For every class $A\in Y$, $w\in\{0,1\}^k$ and function $c:\{0,1\}^k\to X$ there is $c_2\in A$ such that $c_2$ restricted to $\{0,1\}^k\setminus\{w\}$ is equal to $c$.
\end{lemma}

\begin{proof} Let $c_3\in A$ be arbitrary. Let $w'=e^{k+1}_{i,1}(w)$. Let $c_4:\{0,1\}^{k+1}\to X$ be the function defined by $c_4\circ e^{k+1}_{i,0}=c_3$ and $c_4\circ e^{k+1}_{i,1}=c$. By lemma \ref{gencomp} we have that there exists $c_5\in C^{k+1}(X)$ such that it is equal to $c_4$ on the set $\{0,1\}^{k+1}\setminus\{w'\}$. The choice $c_2:=c_5\circ e^{k+1}_{i,1}$ satisfies the requirement of the lemma.
\end{proof}

As a corollary we obtain the next lemma.

\begin{lemma}\label{eqrep} For every class $A\in Y$ and $w\in\{0,1\}^k$ there is a unique $x\in X$ such that $q_{w,x}\in A$.
\end{lemma}

\begin{proof} The existence of $x$ follows directly from lemma \ref{eqrep-gen}. Assume by contradiction that there exist  $x_1,x_2\in X$ with $x_1\neq x_2$ such that $q_{w,x_1}\thickapprox_i q_{w,x_2}$.  Let $c_1\in C^{k+1}(X)$ be such that $q_{w,x_1}=c_1\circ e^{k+1}_{i,0}=c_1\circ e^{k+1}_{i,1}$ and let  $c_2\in C^{k+1}(X)$ be such that $q_{w,x_1}=c_2\circ e^{k+1}_{i,0}$ and $q_{w,x_2}=c_2\circ e^{k+1}_{i,1}$. We have that $c_1(v)=c_2(v)$ holds for every $v\neq e^{k+1}_{i,1}(w)$. Since $X$ is $k$-step this implies that $c_1=c_2$ and thus $x_1=x_2$.
\end{proof}

\begin{corollary}\label{abijection} For every $w\in C^k(X)$ there is a unique bijection $\psi_w:X\to Y$ such that $q_{w,x}\in\psi_w(x)$.
\end{corollary}


\begin{lemma}\label{gencompose} For every pair $A,B\in Y$, $c\in A$ and $j\in [k]$ there is $c_2\in B$ such that $c$ is $j$-composable with $c_2$.
\end{lemma}

\begin{proof} Let $c'=c\circ r_j$ where $r_j$ is the reflection defined in chapter \ref{cubmor}. We have that $c$ and $c'$ are $j$-composable. By lemme \ref{eqrep-gen} there is $c_2\in B$ such that $c_2=c'$ holds on the set $\{0,1\}^k\setminus\{1^k\}$. It is clear that $c_2$ satisfies the requirement of the lemma.
\end{proof}

As an immediate corollary we obtain the following.

\begin{corollary}\label{allcomposable} Every pair $A,B\in Y$ is $j$-composable for every $j\in [k]$.
\end{corollary}


\begin{lemma}\label{refsym} Let $j\in[k]$ be arbitrary and assume that $c\in C^k(X)$ is reflection symmetric in the sense that $c\circ r_j=c$ holds. Then $c$ is $\thickapprox_i$ equivalent with the constant $x_0$ cube in $C^k(X)$.
\end{lemma}

\begin{proof} Let $c_2:\{0,1\}^{k+1}\to X$ be the function defined by the property that $c_2\circ e^{k+1}_{i,1}=c$ and $c_2(v)=x_0$ if $v_i=0$. It is enough to show that $c_2\in C^{k+1}(X)$. Let $c_3:\{0,1\}^k\to X$ be defined by the property that $c_3\circ e^k_{j,1}=c\circ e^k_{j,1}$ and $c_3(v)=x_0$ if $v_j=0$. We have by $k$-ergodicity that $c_3\in C^k(X)$. Let $j'$ be such that $j'=j$ if $j<i$ and $j'=j+1$ if $j\geq i$. Let $p=p_{\{0,1\}^{k+1}\setminus\{j'\}}$ be the projection from $\{0,1\}^{k+1}$ to $\{0,1\}^k$ introduced in chapter \ref{cubmor}. We have that $c_3\circ p=c_2$. Thus by the presheaf axiom we obtain that $c_2\in C^{k+1}(X)$ finishing the proof. 
\end{proof}

Let $E$ denote the $\thickapprox_i$ equivalence class of the constant $x_0$ function in $Y$. 

\begin{lemma}\label{idinv} For every $j\in [k]$ and $c\in C^k(X)$ we have that $c\boxplus_j (c\circ r_j)\in E$. Furthermore if $c\in A\in Y$ then $A\boxplus_j E=A=E\boxplus_j A$.
\end{lemma}

\begin{proof} Since $c\boxplus_j (c\circ r_j)$ is $r_j$ invariant we have the first claim of the lemma by lemma \ref{refsym}. To see that $A\boxplus_j E=A$ let $c_2:\{0,1\}^k\to X$ be the function with the property that $c_2\circ e^k_{j,0}=c_2\circ e^k_{j,1}=c\circ e^k_{j,1}$. We have that $c\boxplus_j c_2=c$ and by lemma \ref{refsym} that $c_2\in E$. The equation $A=E\boxplus_j A$ follows in a similar way. 
\end{proof}

\begin{lemma}\label{grstructure} Let $j\in[k]$. Then the set $Y$ with the operation $\boxplus_j$ is a group with identity element $E$.
\end{lemma}

\begin{proof} Lemma \ref{idinv} verifies the properties of the identity element and the existence of an inverse. It remains to show associativity. For this let $A,B,C$ be elements in $Y$. By iterating lemma \ref{gencompose} we can find elements $c_1\in A,c_2\in B,c_3\in C$ such that $c_1$ is $j$-composable with $c_2$ and $c_2$ is $j$-composable with $c_3$. Then equation  (\ref{comp-assoc}) finishes the proof.
\end{proof}

\begin{lemma} Let $j,l\in [k]$ with $j \neq l$ and $A_{0,0},A_{1,0},A_{0,1},A_{1,0}\in Y$. then  
$$(A_{0,0}\boxplus_j A_{1,0})\boxplus_l (A_{0,1}\boxplus_j A_{1,1})=(A_{0,0}\boxplus_l A_{0,1})\boxplus_j(A_{1,0}\boxplus_l A_{1,1}).$$
\end{lemma}

\begin{proof} To show the statement it is enough to find representatives $c_{a,b}\in A_{a,b}$ for $a,b\in\{0,1\}$ such that 
$$(c_{0,0}\boxplus_j c_{1,0})\boxplus_l (c_{0,1}\boxplus_j c_{1,1})=(c_{0,0}\boxplus_l c_{0,1})\boxplus_j(c_{1,0}\boxplus_l c_{1,1}).$$
For this purpose, let $v_{a,b}$ denote the element in $\{0,1\}^k$ whose $j$-th coordinate is $a$, $l$-th coordinate is $b$ and the rest of the coordinates are $0$. By lemma \ref{eqrep} there are elements $x_{a,b}$ for $a,b\in\{0,1\}$ such that 
$c_{a,b}:=q_{v_{a,b},x_{a,b}}\in A_{a,b}$. It is clear that these cubes satisfies the above equation and so the proof is compplete.
\end{proof}

\begin{theorem}[Eckmann-Hilton \cite{Eckmann1962}]\label{Eckmann-Hilton} Let $\times$ and $\otimes$ be two unital binary operations on a set $X$. Suppose
\[
(a \times b) \otimes (c \times d) = (a \otimes c) \times (b \otimes d)
\]
for all $a, b, c, d \in X$. Then $\times$ and $\otimes$ are in fact the same operation, and are commutative and associative.
\end{theorem}

From the previous lemma and the Eckmann-Hilton argument we obtain that

\begin{corollary}\label{sameop} The operations $\boxplus_j$ on $Y$ are all commutative and they are the same for every $j\in[k]$.
\end{corollary}

In the rest of the chapter we will omit the index $j$ from $\boxplus_j$ when applied to $\thickapprox_i$ classes. 

\begin{corollary} $(Y,\boxplus)$ is an abelian group with identity element $E$.
\end{corollary}

Let us put the abelian group structure $+$ on $X$ by using the bijection $\psi:=\psi_{0^k}:X\to Y$. This means that for $x_1,x_2\in X$ we define $$x_1+x_2=\psi^{-1}(\psi(x_1)\boxplus \psi(x_2)).$$ It follows from this definition that the unit element is $x_0$.

\begin{lemma}\label{corvalues} For $w\in\{0,1\}^k$ and $x\in X$ we have that $\psi_w(x)=(-1)^{h(w)}\psi(x)$.
\end{lemma}

\begin{proof} First we show that if $w_1,w_2\in\{0,1\}^k$ differ only in one coordinate then $\psi_{w_1}(x)\boxplus\psi_{w_2}(x)=E$. Assume that $w_1$ and $w_2$ differ in the $j$-th coordinate and without loss of generality we can assume that the $j$-th coordinate of $w_1$ is $0$. then we have that $q_{w_2,x}\boxplus_j q_{w_1,x}$ is the constant $x_0$ cube. This verifies that $\psi_{w_1}(x)=-\psi_{w_2}(x)$. By iterating this statement on a path connecting $0^k$ with an arbitrary vertex $w$ in the line graph of the cube the general statement follows. 
\end{proof} 

\begin{lemma}\label{kerg-facequiv} Let $c\in C^k(X)$. Then 
$$\psi^{-1}(m(c))=\sum_{v\in\{0,1\}^k}(-1)^{h(v)}c(v)$$ where $m:C^k(X)\to Y$ is the factor map by $\thickapprox_i$.
\end{lemma}

\begin{proof} For each element $v\in\{0,1\}^k$ let $c_v:=q_{v,c(v)}\thickapprox_i \psi_v(c(v))$. By lemma \ref{corvalues} we have that $$\psi^{-1}(m(c_v))=\psi^{-1}(\psi_v(c(v)))=(-1)^{h(v)}c(v).$$ Using this equation, the statement of the lemma follows by applying $\psi^{-1}$ to both sides of the equation
\begin{equation}\label{corsum}
m(c)=\sum_{v\in\{0,1\}^k}m(c_v)
\end{equation} where the summation uses the addition $\boxplus$ on $Y$.
The rest of the proof is the proof of formula (\ref{corsum}). We claim that with an appropriate order of iterated applications of $\boxplus_j$ operations (with changing values for $j$) of cubes $c_v$ we obtain $c$. To do this we define functions $G_r:\{0,1\}^r\to C^k(X)$ for $r\in\{0,1,2,\dots,n\}$ recursively in reverse order starting at $n$. We set $G_n(v):=c_v$ for $v\in\{0,1\}^n$. Assume now that $G_r$ is constructed for some $n\geq r>0$. Then we define $G_{r-1}(v):=G_r(v_0)\boxplus_r G_r(v_1)$ where $v_0$ (resp. $v_1$) is obtained from $v$ by appending $0$ (resp $1$) at the end. It is easy to see that $G_r(v_0)$ and $G_r(v_1)$ are $r$ composable and that at the end of the process we obtain $G_0(0)=c$. By applying $m$ to the whole process and corollary \ref{sameop} the proof is complete.
\end{proof} 

\begin{lemma} We have for a function $c:\{0,1\}^{k+1}\to X$ that $c\in C^{k+1}(X)$ if and only if formula (\ref{ho-abelian-cube}) holds for $c$.
\end{lemma}

\begin{proof} By definition of $\thickapprox_i$ we have that $c\in C^{k+1}(X)$ holds for $c$ if and only if $c_1:=c\circ e^{k+1}_{i,0}, c_2:=c\circ e^{k+1}_{i,1}$ are both in $C^k(X)$ and furthermore $c_1\thickapprox_i c_2$ holds. Observe however that the condition  $c_1,c_2\in C^k(X)$ automatically holds in our case because we assumed $k$-ergodicity for $X$. Thus we have that $c\in C^{k+1}(X)$ holds if and only if $m(c_1)=m(c_2)$. On the other hand notice that by lemma \ref{kerg-facequiv}, the formula (\ref{ho-abelian-cube}) is equivalent with $\psi^{-1}(m(c_1)-m(c_2))=0$. Since $\psi$ is a group isomprphism between $(X,+)$ and $(Y,\boxplus)$ the proof is complete. 
\end{proof}

Now lemma \ref{k-step-det} finishes thew proof of theorem \ref{class-k-step-k-ergodic}.

\section{Characteristic factors of groupspaces}\label{charfact}

In this chapter we introduce relations $\sim_i$ on arbitrary groupspaces that will be used later to reduce the structure of cubespaces to finite step groupspaces. 

\begin{definition} Let $X$ be a groupspace with an equivalence relation $\sim$. Let $m:X\to X/\sim$ be the factor map. The {\bf push-forward cubic structure} on $X/\sim$ is defined by $C^n(X/\sim)=\{m\circ c:c\in C^n(X)\}$.
\end{definition}

It is clear that the push-forward cubic structure satisfies the first two groupspace axioms. 

\begin{definition} Let $X$ be a groupspace. An equivalence relation $\sim$ is a {\bf congruence} of $X$ if $X/\sim$ is a groupspace with the push-forward cubic structure.
\end{definition}

In the next definition we introduce a binary relation and then we show that it is an equivalence relation.

\begin{definition}\label{simdef} Let $X$ be a cubespace and $x,y\in X$. We say that $x\sim_i y$ if there are two cubes $c_1,c_2\in C^{i+1}(X)$ such that $c_1(1^{i+1})=x$,$c_2(1^{i+1})=y$ and $c_1(v)=c_2(v)$ holds for every element in $\{0,1\}^{i+1}_*$.
\end{definition}

It is clear that $\sim_i$ is reflexive and symmetric. It remains to show that it is transitive. We will need the next lemma.

\begin{lemma}\label{equivsimp} Let $X$ be a groupspace and  $i\in\mathbb{N}$. Assume that $x\sim_i y$ holds for $x,y\in X$. Let $c:\{0,1\}^{i+1}\to X$ denote the function with $c(1^{i+1})=x$ and $c(v)=y$ for every $v\in\{0,1\}^{i+1}_*$. Then $c\in C^{i+1}(X)$.
\end{lemma}

\begin{proof} Let $c_1$ and $c_2$ be as in definition \ref{simdef}. By lemma \ref{foldlem} we have that $c_1\circ{\rm fold}_{i+1}$ is cube preserving. Let us define $f:T_{i+1}\to X$ such that $f(\alpha_{i+1}(1^{i+1}))=y$ and $f=c_1\circ{\rm fold}_{i+1}$ holds everywhere else. We claim that $f$ is cube preserving. To see this observe that  $f\circ s_v$ is $c_2$ if $v=\alpha_{i+1}(1^{i+1})$ and is $c_1$ if $v$ is any other element of $T'_{i+1}$. Now lemma \ref{threelem} finishes the proof. 
\end{proof}

As a corollary we obtain the next lemma.

\begin{lemma} Let $X$ be a groupspace and let $i\in\mathbb{N}$. Then $\sim_i$ is an equivalence relation on $X$. 
\end{lemma}

\begin{proof} As we have pointed out, reflexivity and symmetry of $\sim_i$ follows directly from the definition. To see transitivity assume that $x\sim_i y$ and $y\sim_i z$. From symmetry we have that $z\sim_i y$. Let $c_1,c_2:\{0,1\}^{i+1}\to X$ denote the two functions such that $c_1(v)=c_2(v)=y$ holds for $v\neq 1^{i+1}$ and $c_1(1^{i+1})=x,c_2(1^{i+1})=z$. We have by lemma \ref{equivsimp} that both $c_1$ and $c_2$ are in $C^{i+1}(X)$ and so $x\sim_i z$ holds.
\end{proof}

\begin{lemma}\label{finer} Let $X$ be a groupspace and $k\leq i$ be natural numbers. Then $x\sim_i y$ implies $x\sim_k y$. 
\end{lemma}

\begin{proof} Let $c_1$ and $c_2$ be as in definition \ref{simdef} of $\sim_i$. Let $\phi:\{0,1\}^k\to\{0,1\}^i$ be the map defined by $\phi(x_1,x_2,\dots,x_k)=(x_1,x_2,\dots,x_k,1,1,\dots,1)$. Then $c_1\circ\phi,c_2\circ\phi\in C^k(X)$ and they verify according to definition \ref{simdef} that $x\sim_k y$.
\end{proof}

\begin{lemma}\label{eqchange-1} Let $X$ be a groupspace and  $i\in\mathbb{N}$. Assume that $x\sim_i y$ for $x,y\in X$ and that $c\in C^{i+1}(X)$ satisfies $c(0^{i+1})=y$. Let $c':\{0,1\}^{i+1}\to X$ be the function such that $c'(v)=c(v)$ if $v\neq 0^{i+1}$ and $c'(0^{i+1})=x$. Then $c'\in C^{i+1}(X)$.
\end{lemma} 

\begin{proof} Let $c_2:=c\circ{\rm flat}_{i+1}$ and $w=\alpha_{i+1}(0^{i+1})$. Let $f:T_{i+1}\to X$ be such that $f(v)=c_2(v)$ if $v\neq w$ and $f(w)=x$. Notice that $f\circ s_v$ is in $C^{i+1}(X)$ for $v\in T'_{i+1}$ by lemma \ref{equivsimp} and thus $f$ is cube preserving. Observe that $c'=f\circ\alpha_{i+1}$. Then lemma \ref{threelem} finishes the proof.
\end{proof}

This lemma implies the next corollary by symmetry described in lemma \ref{transim}.

\begin{corollary}\label{changequiv} Assume that $x\sim_i y$, $w\in\{0,1\}^{i+1}$ and that $c\in C^{i+1}(X)$ satisfies $c(w)=y$. Let $c':\{0,1\}^{i+1}\to X$ be the function such that $c'(v)=c(v)$ if $v\neq w$ and $c'(w)=x$. Then $c'\in C^{i+1}(X)$.
\end{corollary}

\begin{lemma}\label{changequiv2} Let $X$ be a groupspace, $i\in\mathbb{N}$ and $c\in C^{i+1}(X)$. Then if a function $f:\{0,1\}^{i+1}\to X$ satisfies that $f(v)\sim_i c(v)$ holds for every $v\in\{0,1\}^{i+1}$ then $f\in C^{i+1}(X)$.
\end{lemma}

\begin{proof} The statement follows by repeatedly applying corollary \ref{changequiv} for every vertex $v\in\{0,1\}^{i+1}$ and changing the values of $c$ from $c(v)$ to $f(v)$ step by step. 
\end{proof}

\begin{lemma}\label{changequiv3} Let $X$ be a groupspace, $i\in\mathbb{N}$, $j\leq i+1$ and $c\in C^j(X)$. Then if a function $f:\{0,1\}^j\to X$ satisfies that $f(v)\sim_i c(v)$ holds for every $v\in\{0,1\}^j$ then $f\in C^j(X)$.
\end{lemma}

\begin{proof} The proof is a direct consequence of lemma \ref{down-det} and lemma \ref{changequiv2}.
\end{proof}

\begin{lemma}\label{sim-cubequiv} Let $X$ be a groupspace, $i,n\in\mathbb{N}$. Assume that $c_1,c_2\in C^n(X)$ satisfy that $c_1(v)\sim_i c_2(v)$ holds for every $v$ with $h(v)\leq i$. Then $c_1(v)\sim_i c_2(v)$ holds for every $v\in\{0,1\}^n$.
\end{lemma}

\begin{proof} We go by induction on $n$. If $n\leq i$ then the statement is trivial. If $n=i+1$ then let $c'_2:\{0,1\}^n\to X$ be such that $c_2'(v)=c_1(v)$ for $v\neq 1^n$ and $c'_2(1^n)=c_2(1^n)$. We have that $c'_2(v)\sim_i c_2$ holds for every $v\in\{0,1\}^n$ and thus by lemma \ref{changequiv2} we have that $c_2'\in C^n(X)$. By the definition of $\sim_i$ we have that $c_2'(1^n)\sim_i c_1(1^n)$ and so $c_2(1^n)\sim_i c_1(1^n)$ holds. This proves the claim for $n=i+1$.
Assume that the statement holds for $n-1\geq i+1$. Then, by applying the induction hypothesis to $c_1\circ e^n_{k,0}$ and $c_2\circ e^n_{k,0}$ for every $k\in [n]$ we obtain that $c_1(v)\sim_i c_2(v)$ holds for $h(v)<n$. It remains to show that $c_1(1^n)\sim_i c_2(1^n)$ holds. Let $f:\{0,1\}^{i+1}\to\{0,1\}^n$ be the morphism defined by $f(x_1,x_2,\dots,x_{i+1})=(x_1,x_2,\dots,x_{i+1},1,1,\dots,1)$. Then applying the case $n=i+1$ for $c_1\circ f$ and $c_2\circ f$ the proof is complete. 
\end{proof}

\begin{theorem}\label{eqmain} Let $X$ be a nilspace and let $i\in\mathbb{N}$. Then $\sim_i$ is a groupspace congruence and $X/\sim_i$ is an $i$-step cubespace.  
\end{theorem}

\begin{proof} We denote the factor map $X\to X/\sim_i$ by $m$. To show that $X/\sim_i$ is a groupspace. We need to check the completion axiom. Let $n\in\mathbb{N}$ and let $f:\{0,1\}^n_*\to X/\sim_i$ be such that $f\circ e^n_{j,0}\in C^{n-1}(X/\sim_i)$ for every $j\in [n]$. Let $S$ be the set of elements $v$ in $\{0,1\}^n_*$ with $h(v)\leq i$. We have that $S$ is a simplicial set. Let $f':S\to X$ be such that $m\circ f'=f|_S$ holds. It follows from the assumption on $f$ and lemma \ref{changequiv3} that $f'$ is cube preserving. Then by lemma \ref{simpglue} we obtain that there is $c\in C^n(X)$ such that $c|_S=f'|_S$. It follows by lemma \ref{sim-cubequiv} that $m\circ c\circ e^n_{j,0}(v)= f\circ e^n_{j,0}(v)$ holds for every $v\in\{0,1\}^{n-1}$ and so $(m\circ c)(v)=f(v)$ holds for $v\in\{0,1\}^n_*$. This shows that $m\circ c\in C^n(X)$ is an extension of $f$.

Now we show that $X/\sim_i$ is $i$-step. Let $c_1,c_2\in C^{i+1}(X/\sim_i)$ such that their restrictions to $\{0,1\}^{i+1}_*$ is the same. This means that there are cubes $c_1',c_2'\in C^{i+1}(X)$ such that $m\circ c_1'=c_1$ and $m\circ c_2'=c_2$ and $c_1'(v)\sim_i c_2'(v)$ holds for $v\in\{0,1\}^{i+1}_*$. Let $c_3:\{0,1\}^{i+1}\to X$ be such that $c_3(v)=c_2'(v)$ if $v\in\{0,1\}^{i+1}_*$ and $c_3(1^{i+1})=c_1'(v)$.  Then $c_3(v)\sim_i c_1'(v)$ holds for every $v\in\{0,1\}^{i+1}$ and so by lemma \ref{changequiv2} we have that $c_3\in C^{i+1}(X)$. Since $c_3$ and $c_2'$ agrees on $\{0,1\}^{i+1}_*$ we have by definition that $c_3(1^{i+1})\sim_i c_2'(1^{i+1})$ and thus $c_1(1^{i+1})=c_2(1^{i+1})$.
\end{proof}


\section{Groupspaces as iterated principal bundles}

In this section we show how groupspaces give rise to iterated principal bundles. Let $X$ be a groupspace and for $i\in\mathbb{N}$ let $X_i:=X/\sim_i$. Let $\pi_i:X\to X_i$ denote the factor map from $X$ to $X_i$. Note that $X_0$ is a $1$ point space. We have that $X_i$ is the unique largest $i$-step factor of $X$. Since the equivalence relations $\sim_i$ are increasingly finer as $i$ grows (see lemma \ref{finer}), there are natural factor maps $\pi_{i,j}:X_i\to X_j$ whenever $j\leq i$. These factor maps satisfy $\pi_{j,k}\circ\pi_{i,j}=\pi_{i,k}$ for $k\leq j\leq i$. If $X$ is $k$-step then we have that $\sim_k$ has one element classes and thus $X_k=X$. In this casde we have that $\pi_i=\pi_{k,i}$ holds for $0\leq i\leq k$. 

\begin{lemma}\label{fiber-lem} Let $X$ be a $k$-step groupspace and let $A$ be a $\sim_{k-1}$-class of $X$. Then $A$ is a $k$-step, $k$-ergodic groupspace with the cubic structure $$C^i(A):=\{c~:~c\in C^i(X),c(v)\in A~\forall v\in\{0,1\}^i\}.$$
\end{lemma}

\begin{proof} We start with checking the groupspace axioms. The first two axioms are trivial. Assume that $f:\{0,1\}^i_*\to A$ is a function satisfying the conditions of the completion axiom. We distinguish between two cases. In the first case $i\leq k$. In this case let $a\in A$ be arbitrary and let $c:\{0,1\}^i\to A$ be the function which satisfies $c(1^i)=a$ and $c(v)=f(v)$ for $v\in\{0,1\}^i_*$. We have that $c(v)\sim_{k-1} a$ holds for every $v\in\{0,1\}^i$ and thus applying lemma \ref{changequiv3} for the $i$-dimesional constant $a$ cube and $c$ we obtain that $c\in C^i(A)$ holds. This verifies the completion axiom. For the second case assume that $i>k$. In this case let $c\in C^i(X)$ be the cube guaranteed by the completion axiom in $X$ when applied to $f$. Notice that $f\circ\pi_{k-1}$ is a constant function in $C^i(X_{k-1})$ and thus, since $X/\sim_{k-1}$ is $k-1$-step we have that $c\circ\pi_{k-1}$ is also constant. This proves that $c\in C^i(A)$ holds and the completion axiom is proved. 

The fact that $A$ is $k$-step follows directly from the fact that $X$ is $k$-step and $A$ has fewer cubes. To see that $A$ is $k$-ergodic observe that every map $c:\{0,1\}^k\to A$ is congruent with the constant $a\in A$ cube modulo $\sim_{k-1}$. It follows from lemma \ref{changequiv3} that $c\in C^k(A)$.
\end{proof}

\begin{definition} Let $X$ be a $k$-step groupspace. A fiber of $X$ is an equivalence class $F$ of $\sim_{k-1}$ with the inherited cubic structure defined in lemma \ref{fiber-lem}.
\end{definition}

\begin{lemma}\label{eqcubeinfiber} Let $X$ be a $k$-step groupspace for $k\geq 2$ and let $F_1,F_2$ be two fibers of $X$. Let $c_1\in C^k(F_1)$. Then there exists $c_2\in C^k(F_2)$ such that $c_1\thickapprox_1 c_2$ holds.
\end{lemma}

\begin{proof} Let $x_0\in F_1$ be an arbitrary element. We have by lemma \ref{eqrep} that there exists $x\in F_1$ such that $q_{1^k,x}\thickapprox_1 c_1$ holds. Now let $x_0'\in F_2$ be arbitrary. Let $f:\{0,1\}^{k+1}_*\to X$ be such that $f\circ e^{k+1}_{1,0}=q_{1^k,x}$ and $f(v)=x_0'$ for every other elements of $\{0,1\}^{k+1}_*$. It is easy to see that $f$ is a corner and so there is a completion $c:\{0,1\}^{k+1}\to X$ in $C^{k+1}(X)$. By factoring with $\sim_{k-1}$ and using that $X_{k-1}$ is $k-1$-step we obtain that $c(1^{k+1})\sim_{k-1} x_0'$ and thus $c(1^{k+1})\in F_2$. It follows that $c\circ e^{k+1}_{1,1}$ is a good choice for $c_2$.
\end{proof}

The main theorem of this chapter is the following.

\begin{theorem}\label{fiberisomorphism} Let $X$ be a $k$-step groupspace for $k\geq 2$. Then all fibers of $X$ are isomorphic.
\end{theorem}

\begin{proof} Lat $F_1$ and $F_2$ be two fibers of $X$. According to lemma \ref{fiber-lem} we have that $F_1$ and $F_2$ are both $k$-step, $k$-ergodic groupspaces. Theorem \ref{class-k-step-k-ergodic} says that $F_i$ as a groupspace is isomorphic to $\mathcal{D}_k(A_i)$ for some abelian groups $A_1$ and $A_2$. It suffices to show that $A_1$ and $A_2$ are isomorphic. Let $Y_i=C^k(A_i)/\thickapprox_1$. Lemma \ref{grstructure} describes the structure of $A_i$ as $(Y_i,\boxplus_1)$. First we define a bijection $\phi$ between $Y_1$ and $Y_2$ and then show that it is an isomorphism. Let $c_i\in C^k(A_i)$ be cubes for $i=1,2$. Note that $c_i\in C^k(X)$ holds for $i=1,2$. We say that $\phi(c_1)=c_2$ if $c_1\thickapprox_1 c_2$ holds in $C^k(X)$. The fact that $\thickapprox_1$ is an equivalence relation and lemma \ref{eqcubeinfiber} shows $\phi$ is a well defined function and that it is a bijection. Let $c_{i,j}\in C^k(A_i)$ for $i,j\in\{1,2\}$ and assume  that $c_{i,1}$ is $1$-composable with  $c_{i,2}$ for $i=1,2$ and that $c_{1,i}\thickapprox_1 c_{2,i}$ for $i=1,2$. Then lemma \ref{compwell} implies that the cubes $c_i:=c_{i,1}\boxplus_1 c_{i,2}$ satisfy $c_1\thickapprox_1 c_2$ and thus $\phi(c_1)=c_2$. This verifies that $\phi$ is an isomorphism between $Y_1$ and $Y_2$. 
\end{proof}

In theorem \ref{fiberisomorphism} we obtain the isomorphism of the different fibers of a $k$-step groupspace. We are ready to define the sequence of the structure groups of an arbitrary groupspace $X$.

\begin{definition}[Structure groups of groupspaces] Let $X$ be an arbitrary groupspace. We have by theorem \ref{class1step} that $X/\sim_1=X_1$ (Which is its own fiber) is isomorphic to $\mathcal{D}_1(G_1)$ for some possibly non-commutative group $G_1$. Furthermore if $k\geq 2$, the fibers of $X/\sim_k=X_k$ are all isomorphic to $\mathcal{D}_k(G_k)$ for some abelian group $G_k$. The group $G_k$ (for $k=1,2,\dots$) is called the {\bf $k$-th structure group of $X$}.
\end{definition}

Now we describe how to view a $k$-step groupspace $X$ as a principal bundle over its $k-1$ step factor $X_{k-1}$. If $k=1$ then $X\simeq\mathcal{D}_1(G)$ for some general group. Both the right and the left multiplication action of $G$ on $\mathcal{D}_1(G)$ is a group of groupspace automorphisms and thus $\mathcal{D}_1(G)$ is naturally a principle $G$ bundle over the one point space. Assume that $k\geq 2$. Theorem \ref{fiberisomorphism} says that each fiber of $X$ is isomorphic to $\mathcal{D}_k(A)$ for some abelian group $A$. However this does not yet yield a canonical action of $A$ on $X$ whose orbits are the fibers. To obtain such an action, observe that the proof of theorem \ref{fiberisomorphism} yields a cannonical isomorphism $\alpha_F$ between $Y_F:=C^k(F)/\thickapprox_1$ and an abelian group $A$ for each fiber $F$. Recall that $(Y_F,\boxplus)$ is an abelian group and the map $\phi$ constructed in theorem \ref{fiberisomorphism} is canonical. To define the action of $A$ on a fiber $F$ let $a\in A$. Then if $x\in F$ we have that there exists a cube $c\in C^k(F)$ such that $c(v)=x$ holds for $v\neq 0^k$ and furthermore $c\thickapprox_1 a$ where $a$ is identified with its image in $Y_F$ under the canonical isomorphism. We define $x^a$ as $c(0^k)$. It is easy to see that this is an action of $A$. To see this let us chose an arbitrary element $x_0\in F$ and let $\psi:F\to Y_F$ be as in Chapter \ref{Class-k}.  It follows from lemma \ref{kerg-facequiv} that $\psi(y)-\psi(x)=a$ and thus $y=x+\psi^{-1}(a)$. Notice that both $\psi$ and the $+$ operation depends on $x_0$. However the action obtained this way does not depend on $x_0$ since the definition of $y$ did not use $x_0$. It follows that the action $x\mapsto x+\psi^{-1}(a)$ is canonical.

\end{document}